\documentclass[3p,12pt]{elsarticle}
\usepackage{color}
\usepackage[colorlinks,linkcolor=blue,citecolor=blue]{hyperref}
\usepackage{amsmath,amsfonts,amsmath,amssymb,amsthm}
\usepackage{mathrsfs} 
\usepackage{graphicx,epstopdf,float,subfigure}
\usepackage[margin=10pt,font=small,labelfont=bf,labelsep=period]{caption}

\usepackage{multirow}

\numberwithin{equation}{section} 
\newtheorem{theorem}{Theorem}[section]
\newtheorem{lemma}{Lemma}[section]

\newtheorem{remark}{Remark}[section]
\newtheorem{example}{Example}[section]
\biboptions{sort&compress}

\newtheoremstyle{mythm}{1.5ex plus 1ex minus .2ex}{1.5ex plus 1ex minus .2ex}
{\song}{\parindent}{\song\bfseries}{}{1em}{}
\theoremstyle{mythm}

\begin{document}

\begin{frontmatter}

\title{A sparse grid discrete ordinate discontinuous Galerkin method for the radiative transfer equation}
\date{Version 0.0, June 12, 2019}
\author{Jianguo Huang\footnote{Corresponding author.}}
\ead{jghuang@sjtu.edu.cn}
\author{Yue Yu}
\ead{terenceyuyue@sjtu.edu.cn}
\address{School of Mathematical Sciences, and MOE-LSC, Shanghai Jiao Tong University\\
Shanghai 200240, China}

\begin{abstract}
  The radiative transfer equation is a fundamental equation in transport theory and applications, which is a 5-dimensional PDE in the stationary one-velocity case, leading to great difficulties in numerical simulation. To tackle this bottleneck, we first use the discrete ordinate technique to discretize the scattering term, an integral with respect to the angular variables, resulting in a semi-discrete hyperbolic system. Then, we make the spatial discretization by means of the discontinuous Galerkin (DG) method combined with the sparse grid method. The final linear system is solved by the block Gauss-Seidal iteration method. The computational complexity and error analysis are developed in detail, which show the new method is more efficient than the original discrete ordinate DG method. A series of numerical results are performed to validate the convergence behavior and effectiveness of the proposed method.
\end{abstract}


\end{frontmatter}

\section{Introduction}

Radiation transport is a physical process of energy transfer in the form of electromagnetic radiation which is affected by absorption, emission and scattering as it passes through the background materials. The radiative transfer equation (RTE) is an important mathematical model used to describe these interactions, finds applications in a wide variety of subjects, including neutron transport, heat transfer, optics, astrophysics, inertial confinement fusion, and high temperature flow systems, see for examples  \cite{Han-Huang-Eichholz-2010,Case-Zweifel-1967,Duderstadt-Martin-1978,Agoshkov1998,Golse-Jin-Levermore-1999,Tang-2009}.

The RTE can be viewed as a hyperbolic-type integro-differential equation. Even for the stationary monochromatic RTE, it is five-dimensional in the phase space, and hence cannot have a closed-form solution in general.  Thus, the numerical solution of the equation is unavoidable and critical in applications. In history, the Monte-Carlo method is a typical approach for numerical simulation (cf. \cite{Caflisch-1998} and the references therein). The advantage is its simplicity and dimension-free convergence, and the weakness is its heavy computational cost and slow convergence. Until now, there have developed many other numerical methods as well. For the angular discretization, the typical methods include discrete ordinate methods (or $S_N$ methods) and spherical harmonic methods (or $P_N$ method); for the spatial discretization, the typical methods include finite difference methods, finite element methods and spectral methods. We refer to \cite{Lewis-Miller-1984,Balsara-2001,Atkinson-Han-2012,Han-Huang-Eichholz-2010,
Case-Zweifel-1967,Frank-Klar-Larsen-2007,Larsen-Morel-2010,Golse-Jin-Levermore-1999} for details. Due to the flexibility and easy implementation, the discrete ordinate method is frequently used for angular discretization in practice. If the spatial domain is regular, this semi-discrete method is further discretized by the Chebyshev spectral method in \cite{Kim-Moscoso-2002,Edstrom-2005,Asadzadeh-Kadem-2006} and the meshless discretization in \cite{Sadat-2006,Kindelan-Bernal-2010,Wang-Sadat-Tan-2014}. In recent years, the positivity-preserving schemes are also developed very technically in \cite{Yuan-Cheng-Shu-2016,Dan-Cheng-Shu-2018,Zhang-Cheng-Qiu-2019}.
For numerical solvers such as source iteration and multigrid algorithms, one can refer to \cite{Chang-Manteuffel-2007,Adams-Larsen-2002,Sheng-Wang-Han-2016,Shao-Sheng-Wang-2020}.

On the other hand, except the Monte-Carlo method, all the methods mentioned above solve the problems with reduced dimensions. In this paper, we intend to attack the problem in its original form with 3-spatial variables and 2-angular variables. In this case, most usual methods suffer from the so-called ``the curse of dimension", which indicates the low rate of convergence in terms of number of degrees of freedom due to the high dimensionality of the underlying problem. To the best of our knowledge, the sparse grid method, also called the sparse tensor product method, is an effective way to overcome the bottleneck. Historically, the idea of sparse grids can be traced back to Smoljak's construction of multivariate quadrature formulas using combinations of tensor products of suitable one-dimensional formulas (cf. \cite{Smoljak-1963,Gerstner-Griebel-1998}). More recently, the systematic and thorough studies on the method can be found in \cite{Zenger-1990,Griebel-1991,Griebel-1998,Gerstner-Griebel-1998}. In addition, several sparse grid methods are devised in \cite{Widmer-Hiptmair-Schwab-2008,Grella-Schwab-2011} for solving the RTE through conforming spatial discretization. However, according to the computational experience, it is preferable to use the discontinuous Galerkin (DG) method for spatial discretization for hyperbolic problems (cf. \cite{Brezzi-Cockburn-Marini-2006,Cockburn-2003,Brezzi-Marini-Suli-2004}), in order to capture non-smooth physical solutions. In \cite{Wang-Tang-Guo-2016}, the sparse grid technique combined with the DG method has been developed for elliptic equations. This method is also applied to transport equations in \cite{Guo-Cheng-2016,Guo-Cheng-2017}, but the scattering effect is not considered. The adaptive analogues of their methods are also given in \cite{Guo-Cheng-2017,Tao-Jiang-Cheng-2019}.

In this paper, we are intended to propose and analyze a sparse grid DG method to solve the RTE, following the ideas in \cite{Han-Huang-Eichholz-2010} and \cite{Guo-Cheng-2016}. Unlike the studies in \cite{Widmer-Hiptmair-Schwab-2008,Grella-Schwab-2011}, the DG method will be used to carry out the spatial discretization. And different from \cite{Guo-Cheng-2016}, we will discuss in detail the efficient solution of the 5-dimensional RTE with scattering effect. Concretely speaking, the discrete ordinate technique is first applied to discretize the scattering term, an integral with respect to the angular variables, by simply picking several directions spanning the solid angle, resulting in a semi-discrete coupled hyperbolic system. In view of the hyperbolic nature of the semi-discrete system, the DG method is further employed for spatial discretization, yielding a fully discrete method. To overcome the curse of dimension, the sparse DG space is constructed by using the techniques in wavelet analysis to replace the original piecewise polynomial approximation space. We achieve the complexity analysis and error analysis of the method using some arguments in \cite{Han-Huang-Eichholz-2010} and \cite{Guo-Cheng-2016}, which show the new approach can greatly reduce the spatial degrees of freedom while keeping almost the same accuracy up to multiplication of an $\log$ factor. For the resulting linear system, considering its block structure, we solve it using the block Gauss-Seidal iteration method. A series of numerical examples are reported to validate the accuracy and performance of the proposed method. Furthermore, we also extend the method to solve the RTE efficiently for some non-tensor product spatial domains in two dimensions.

We end this section by introducing some notations and symbols frequently used in this paper.  For a bounded Lipschitz domain $D$, the symbol $( \cdot , \cdot )_D$ denotes the $L^2$-inner product on $D$, $\|\cdot\|_{0,D}$ denotes the $L^2$-norm, and $|\cdot|_{s,D}$ is the $H^s(D)$-seminorm. For all integer $k\ge 0$, $\mathbb{P}_k(D)$ is the set of polynomials of degree $\le k$ on $D$.

The jumps and averages for scalar and vector-valued functions ($v,\boldsymbol \tau $, respectively) on an edge $e$ common to two elements $K_1,K_2$ are defined by
\[ [\![ v ]\!] = v_1\boldsymbol{n}_1 + v_2\boldsymbol{n}_2,\quad \{\!\!\{v\}\!\!\} = \frac{v_1 + v_2}{2},\]
\[ [\![\boldsymbol{\tau }]\!] = \boldsymbol{\tau }_1 \cdot \boldsymbol{n}_1 + \boldsymbol{\tau }_2 \cdot \boldsymbol{n}_2,\quad\{\!\!\{\boldsymbol{\tau }\}\!\!\} = \frac{\boldsymbol{\tau }_1 + \boldsymbol\tau_2}{2},\]
where $\boldsymbol{n}_1,\boldsymbol{n}_2$ are the unit outward normals to $K_1,K_2$, respectively. On a boundary edge or face, $[\![ v ]\!] = v\boldsymbol{n}$ and $\{\!\!\{ \boldsymbol\tau \}\!\!\} = \boldsymbol\tau$.
 Moreover, for any two quantities $a$ and $b$, ``$a\lesssim b$" indicates ``$a\le C b$" with the hidden constant $C$ independent of the mesh size $h_K$, and ``$a\eqsim b$" abbreviates ``$a\lesssim b\lesssim a$".

\section{Radiative transfer equation}\label{sec:RTE}

The steady-state monoenergetic version of the radiative transfer equation is expressed as (cf. \cite{Han-Huang-Eichholz-2010,Atkinson-Han-2012})
\begin{equation}\label{RTE2}
\boldsymbol\omega  \cdot \nabla u(\boldsymbol x,\boldsymbol\omega ) + \sigma_t(\boldsymbol x)u(\boldsymbol x,\boldsymbol\omega ) = \sigma_s(\boldsymbol x)(Su)(\boldsymbol x,\boldsymbol\omega ) + f(\boldsymbol x,\boldsymbol\omega ),\quad \boldsymbol x\in D,\boldsymbol\omega \in S^2 .
\end{equation}
Here, $D$ is a domain in $\mathbb R^3$ and $S^2$ denotes the unit sphere in $\mathbb R^3$,
$u(\boldsymbol x,\boldsymbol\omega )$ is a function of three space variables $\boldsymbol x$ and two angular variables $\boldsymbol\omega$, $\sigma_t = \sigma_a + \sigma_s$ with $\sigma_a$ being the macroscopic absorption cross section, and $\sigma_s$ the macroscopic scattering cross section, and $f$ is a source function in $D$.
We impose an inflow boundary value condition
\begin{equation}\label{inflowbd}
u(\boldsymbol x,\boldsymbol\omega)=\alpha(\boldsymbol x,\boldsymbol\omega),\quad (\boldsymbol x,\boldsymbol\omega) \in {\Gamma_ - },
\end{equation}
where $\Gamma_ - $ is defined by
\begin{equation}\label{RTEbound}
\Gamma_ -  = \{ (\boldsymbol x,\boldsymbol\omega ): \boldsymbol x \in \partial D,~\boldsymbol\omega  \in S^2 ,\quad \boldsymbol n(\boldsymbol x) \cdot \boldsymbol\omega  < 0 \}.
\end{equation}
The symbol $S$ on the right-hand side of \eqref{RTE2} is an integral operator defined by
\begin{equation}\label{intop}
(Su)(\boldsymbol x,\boldsymbol\omega ) = \int_{S^2}  g(\boldsymbol x,\boldsymbol\omega  \cdot \hat{\boldsymbol\omega} )u(\boldsymbol x,\hat{\boldsymbol\omega} ) {\rm d}\sigma (\hat{\boldsymbol\omega} )
\end{equation}
with $g$ being a nonnegative normalized phase function
\[\int_{S^2} g(\boldsymbol x, \boldsymbol\omega  \cdot \hat{\boldsymbol\omega }) {\rm d}\sigma (\hat{\boldsymbol\omega }) = 1,\quad \boldsymbol x\in D, \boldsymbol\omega \in S^2.\]
In most applications, the function $g$ is assumed to be independent of $\boldsymbol x$. One well-known example considered in this paper is the Henyey-Greenstein phase function
 \begin{equation}\label{Heny}
g(t) = \frac{1 - \eta ^2}{4\pi (1 + \eta ^2 - 2 \eta t)^{3/2}},~~t \in [ - 1,1],
 \end{equation}
where the parameter $\eta \in (-1,1)$ is the anisotropy factor for the scattering medium which measures the strength of forward peakedness of the phase function. Note that $\eta =0$ for isotropic scattering, $\eta >0$ for forward scattering, and $\eta <0$ for backward scattering.

We assume that
\begin{itemize}
  \item $\sigma_t,\sigma_s \in L^\infty (D)$, $\sigma_s \ge 0$ a.e. in $D$, $\sigma_a = \sigma_t - \sigma_s \ge c_0$ in $D$ for a constant $c_0>0$.
  \item $f(\boldsymbol x,\boldsymbol\omega)\in L^2(D\times S^2)$ and is a continuous function with respect to $\boldsymbol\omega \in S^2$.
\end{itemize}
Under these assumptions, the problem \eqref{RTE2}-\eqref{inflowbd} has a unique solution $u \in H_2^1(D \times S^2)$ (cf. \cite{Han-Huang-Eichholz-2010}), where
\[H_2^1(D \times S^2): = \{ v \in L^2(D \times S^2): \boldsymbol\omega \cdot \nabla v \in L^2(D \times S^2)\}.\]

\section{The sparse grid discrete-ordinate DG method for the RTE}\label{sec:spdodg}
In this section, we first recall the construction of sparse discontinuous finite element spaces; One can refer to \cite{Wang-Tang-Guo-2016,Alpert-1993,Alpert-Beylkin-Gines-2002} and the references therein for details. Then, we will present in detail the sparse grid discrete-ordinate DG method for the RTE.

\subsection{Construction of sparse DG spaces}\label{sec:sdg}
Let $\Omega  = [0,1]$ and partition it into $2^n$ cells with uniform cell size $h = 2^{ - n}$. The resulting $n$-th level grid is denoted by ${\Omega_n}$ and the $j$-th cell is given by
\[I_n^j = (2^{ - n}j,2^{ - n}(j + 1)],\quad j = 0,1, \cdots ,2^n - 1.\]
We define
\[V_n^k = \{ v: v |_{I_n^j} \in \mathbb P_k(I_n^j), \quad j = 0,1, \cdots ,2^n - 1\}\]
to be the piecewise polynomial space on ${\Omega_n}$. One can check that there exists the nested structure for different values of $n$: $V_0^k \subset V_1^k \subset  \cdots  \subset V_n^k \subset  \cdots$.
Denote $W_n^k$ to be the orthogonal complement of $V_{n-1}^k$ in $V_n^k$ with respect to the $L^2(\Omega )$ inner product, i.e.,
\[V_{n - 1}^k \oplus W_n^k = V_n^k,\quad W_n^k \bot V_{n - 1}^k, \quad n \ge 1,\]
where for simplicity set $W_0^k = V_0^k$. We then obtain an orthogonal decomposition of the DG space
\[V_N^k = \mathop  \bigoplus \limits_{0 \le n \le N} W_n^k.\]

We proceed to review the construction in multi-dimensions. For $\Omega  = [0,1]^d$, let $h_m = 2^{ - n_m}$ be the step size along $x_m$-direction. For simplicity, we use the notations of multi-indices in the following. Let $\boldsymbol n = (n_1,n_2, \cdots ,n_d)$. Then the cell size can be denoted by
\[h_{\boldsymbol n} = (2^{ - n_1},2^{ - n_2}, \cdots ,2^{ - n_d}) = 2^{ - \boldsymbol n}\]
and the associated grid is written by $\Omega_{\boldsymbol n}$ whose $\boldsymbol j$-th cell is given by
\[I_{\boldsymbol n}^{\boldsymbol j} = I_{n_1}^{j_1} \times I_{n_2}^{j_2} \times  \cdots  \times I_{n_d}^{j_d},\quad \boldsymbol j = (j_1,j_2, \cdots ,j_d),\]
where
\[I_{n_m}^{j_m} = (2^{ - n_m}j_m,2^{ - n_m}(j_m + 1)],\quad j_m = 0,1, \cdots ,2^{n_m} - 1\]
is the element along $x_m$-axis. With multi-indices notation we have $\boldsymbol 0\le \boldsymbol j \le2^{\boldsymbol n}-\boldsymbol 1$.  Introduce a tensor-product piecewise polynomial space as
\[\boldsymbol V_{\boldsymbol n}^k = \{ \boldsymbol v: \boldsymbol v(\boldsymbol x) \in \mathbb Q_k(I_{\boldsymbol n}^{\boldsymbol j}), \quad \boldsymbol 0 \le \boldsymbol j \le 2^{\boldsymbol n} - \boldsymbol 1 \},\]
where $\mathbb Q_k(I_{\boldsymbol n}^{\boldsymbol j})$ consists of polynomials of degree up to $k$ in each dimension on cell $I_{\boldsymbol n}^{\boldsymbol j}$. If we use an equal refinement of size $h: = h_N = 2^{-N}$ in each coordinate
direction, the grid and space will be denoted by $\Omega_N$ and $\boldsymbol V_N^k$, respectively. With the usual convention, we also use $\mathcal{T}_h$ and $V_h^k$ instead.

It is obvious that
\[\boldsymbol V_{\boldsymbol n}^k = V_{n_1}^k \times V_{n_2}^k \times  \cdots  \times V_{n_d}^k.\]
We similarly define the tensor-product multiwavelet space as
\[\boldsymbol W_{\boldsymbol n}^k = W_{n_1}^k \times W_{n_2}^k \times \cdots \times W_{n_d}^k.\]
Observing the fact that
\[V_{n_m}^k = \mathop  \bigoplus \limits_{0 \le j_m \le n_m} W_{j_m}^k,\]
we have the following expansion
\[\boldsymbol V_{\boldsymbol n}^k = \mathop  \bigoplus \limits_{\boldsymbol 0 \le \boldsymbol j \le \boldsymbol n} \boldsymbol W_{\boldsymbol j}^k,\quad ~
\boldsymbol V_N^k = \mathop  \bigoplus \limits_{| \boldsymbol j |_\infty  \le N} \boldsymbol W_{\boldsymbol j}^k.\]
The sparse finite element approximation space on $\Omega_N$ is defined by the following truncated space
\[
\widehat{\boldsymbol V}_N^k: = \mathop  \bigoplus \limits_{| \boldsymbol n |_1 \le N} \boldsymbol W_{\boldsymbol n}^k,\quad | \boldsymbol n |_1 = n_1 + n_2 +  \cdots  + n_d.
\]
The number of degrees of freedom of sparse DG space is $\mathcal{O}( h^{ - 1}| \log_2h |^{d - 1} )$ with $h = 2^{ - N}$,
which is significantly less than that of DG space with exponential dependence on $d$.

\subsection{The sparse grid discrete-ordinate DG method}
For any continuous function $F(\boldsymbol\omega)$ defined on the unit sphere $S^2$, we write the numerical quadrature to be used in the form
\begin{equation}\label{jiaoduli}
\int_{S^2} F(\boldsymbol\omega ) {\rm d}\sigma (\boldsymbol\omega ) \approx \sum\limits_{l = 1}^L {w_lF(\boldsymbol\omega ^l)} ,\quad \boldsymbol\omega ^l \in S^2 ,~1 \le l \le L.
\end{equation}
The integral operator $S$ is then approximated by
\begin{equation}\label{Suapp}
(Su)(\boldsymbol x,\boldsymbol\omega ) \approx (S_du)(\boldsymbol x,\boldsymbol\omega ) := \sum\limits_{l = 1}^L w_lg(\boldsymbol x,\boldsymbol\omega  \cdot \boldsymbol\omega ^l)u(\boldsymbol x,\boldsymbol\omega ^l) .
\end{equation}
Regarding the accuracy of the quadrature \eqref{jiaoduli}, we will write $n$ for the algebraic precision, i.e., the quadrature integrates exactly all spherical polynomials of total degree no more than $n$ and does not integrate exactly some spherical polynomial of total degree $n+1$. Then we have the following estimate (cf. \cite{Han-Huang-Eichholz-2010})
\begin{equation}\label{spest}
\Big| \int_{S^2}  F (\boldsymbol\omega){\rm d}\sigma (\boldsymbol\omega) - \sum\limits_{l = 1}^L w_l F( \boldsymbol\omega ^l ) \Big| \le c_sn^{ - s}\| F \|_{s,S^2 },\quad F \in H^s({S^2} ),\quad s > 1,
\end{equation}
where $c_s$ is a universal constant depending only on $s$.
Associated with the numerical quadrature, we further define
\begin{equation}\label{mx}
m(\boldsymbol x) = \mathop {\max }\limits_{1 \le i \le L} \sum\limits_{l = 1}^L w_lg(\boldsymbol x,\boldsymbol\omega ^l \cdot \boldsymbol\omega ^i)
\end{equation}
and make the following assumption (cf. \cite{Han-Huang-Eichholz-2010}):
\begin{equation}\label{assm}
 \sigma_t-m\sigma_s \ge c_0' ~\mbox{in $D$ for some constant $c_0'>0$}.
\end{equation}

Using the quadrature \eqref{Suapp}, we can discretize \eqref{RTE2} in angular direction to get
\begin{equation}\label{semibs}
\boldsymbol\omega ^l \cdot \nabla u^l + \sigma_tu^l = \sigma_s(\boldsymbol x)\sum\limits_{i = 1}^L w_ig(\boldsymbol x,\boldsymbol\omega ^l \cdot \boldsymbol\omega ^i)u^i  + f^l,\quad 1 \le l \le L
\end{equation}
with boundary value condition
\begin{equation}\label{seimibond}
u^l(\boldsymbol x) = \alpha^l(\boldsymbol x),\quad (\boldsymbol x,\boldsymbol\omega ^l)\in \Gamma_-,~~1 \le l \le L,
\end{equation}
where $u^l =u^l(\boldsymbol x)$ is the approximation to $u(\boldsymbol x,\boldsymbol\omega ^l)$.

The system \eqref{semibs} is a first-order hyperbolic problem in space, which will be further discretized by DG method. Let $\{ \mathcal{T}_h \}_{h > 0}$ be a regular family of triangulations of $D$. Assume that $\mathcal{E}_h$ consists of the set of all edges ($d=2$) or faces ($d=3$) in $\mathcal{T}_h$ and $\mathcal{E}_h^0$ the set of all interior edges or faces. By a direct manipulation, we obtain the following identity (cf. \cite{Brezzi-Marini-Suli-2004}):
\begin{lemma}\label{dgforweak}
 For $(\varphi,\boldsymbol\tau ) \in H^s(\mathcal{T}_h) \times [H^s(\mathcal{T}_h)]^d$, $s > 1/2$, there holds
\begin{equation}\label{dgformula}
 \sum\limits_{K \in \mathcal{T}_h} \int_{\partial K} \varphi\boldsymbol\tau  \cdot \boldsymbol n {\rm d}s  = \sum\limits_{e \in \mathcal{E}_h} \int_e \{\!\!\{\boldsymbol\tau  \}\!\!\} \cdot [\![ \varphi ]\!] {\rm d}s + \sum\limits_{e \in \mathcal{E}_h^0} \int_e \{ \!\!\{ \varphi \}\!\!\} \cdot [\![ \boldsymbol\tau ]\!] {\rm d}s.
\end{equation}
Further, if $u\in H^s(\omega _e)$ and $s>1/2$, then we have the following weak continuity
\[\int_e [\![ u ]\!]v {\rm d}s = 0,\quad v \in L^2(e),\quad e \in \mathcal{E}_h^0,\]
 where $\omega _e$ is the set of elements sharing $e$ as an edge ($d=2$) or faces ($d=3$).
\end{lemma}

We define a discontinuous finite element space by
\begin{equation}\label{DGsingle}
V_h = \Big\{ v \in L^2(D ): v|_K \in \mathbb P_k(K),~~K \in \mathcal{T}_h \Big\},
\end{equation}
where $\mathbb P_k(K)$ denotes the set of all polynomials on $K$ with degree $\le k$. Multiplying \eqref{semibs} by any $v_h \in V_h$, we obtain from the integration by parts that
\begin{align*}
  \sum\limits_{K \in \mathcal{T}_h} \Big[ \int_K (  - u^l(\boldsymbol\omega ^l \cdot \nabla v_h) + \sigma_tu^lv_h) {\rm d}x + \int_{\partial K} (\boldsymbol\omega ^l \cdot \boldsymbol n)u^lv_h {\rm d}s \Big]   \\
  \quad \quad  = \int_D \sigma_s\sum\limits_{i = 1}^L w_ig( \cdot ,\boldsymbol\omega ^l \cdot \boldsymbol\omega ^i)u^i v_h {\rm d}x + \int_D f^lv_h {\rm d}x,~~1 \le l \le L.
\end{align*}
 Taking $\boldsymbol\tau = \boldsymbol\omega^lu^l$ and $\varphi  = v_h$ in \eqref{dgformula}, we immediately obtain the following system
\[a_h^{(l)}(u^l,v_h) + b_h^{(l)}(u^l,v_h) = (f^l,v_h) + \langle \alpha ^l,v_h\rangle ^{(l)},\quad v_h \in V_h,\]
where
\begin{align}\label{ahl}
a_h^{(l)}(u^l,v_h) &= \sum\limits_{K \in \mathcal{T}_h} \int_K (  - u^l(\boldsymbol\omega ^l \cdot \nabla v_h) + \sigma_tu^lv_h ) {\rm d}x  \nonumber \\
 & \quad - \int_D \sigma_s\sum\limits_{i = 1}^L w_ig( \cdot ,\boldsymbol\omega ^l \cdot \boldsymbol\omega ^i)u^i v_h {\rm d}x,
\end{align}
\begin{equation}\label{bhl}
b_h^{(l)}(u^l,v_h) = \sum\limits_{e \not\subset \Gamma_ - } \int_e \{\!\! \{ \boldsymbol\omega ^lu^l \} \!\!\} \cdot [\![ v_h
 ]\!]  {\rm d}s,
\end{equation}
\begin{equation}\label{flvh}
(f^l,v_h) = \int_D f^lv_h {\rm d}x,\quad \langle \alpha ^l,v_h\rangle ^{(l)} =  - \sum\limits_{e \subset \Gamma_ - } \int_e \boldsymbol\omega ^l \cdot \boldsymbol n\alpha ^lv_h  {\rm d}s.
\end{equation}
Define $\boldsymbol V_h = ( V_h )^L$ and write a generic element as $\boldsymbol v_h: = \{ v_h^l \}_{l = 1}^L$. The global formulation can be expressed as
\[\sum\limits_{l = 1}^L w_l( a_h^{(l)}(u^l,v_h^l) + b_h^{(l)}(u^l,v_h^l) )  = \sum\limits_{l = 1}^L w_l( (f^l,v_h^l) + \langle \alpha ^l,v_h^l\rangle ^{(l)} ) .\]
Then the discrete-ordinate DG method is: Find $\boldsymbol u_h: = \{u_h^l\} \in \boldsymbol V_h$ such that
\begin{equation}\label{unstabilized}
a_h(\boldsymbol u_h,\boldsymbol v_h) = F(\boldsymbol v_h),\quad \boldsymbol v_h \in \boldsymbol V_h,
\end{equation}
where
\[a_h(\boldsymbol u_h,\boldsymbol v_h) = \sum\limits_{l = 1}^L w_l( a_h^{(l)}(u_h^l,v_h^l) + b_h^{(l)}(u_h^l,v_h^l) ) ,\]
\[F(\boldsymbol v_h) = \sum\limits_{l = 1}^L w_l( (f^l,v_h^l) + \langle \alpha ^l,v_h^l\rangle^{(l)} ) .\]

It is preferable to add some stabilization terms in the DG scheme to penalize the jump of the solution across interior edges or faces of the triangulation. One approach introduced in \cite{Brezzi-Marini-Suli-2004} is to replace the average $\{\!\! \{ \boldsymbol\omega ^lu^l \} \!\!\}$ in \eqref{bhl} by $\{\!\! \{ \boldsymbol\omega^l u_h^l \} \!\!\} + c_e^l[\![ u_h^l ]\!]$, where $c_e^l$ is a nonnegative function over $e$ satisfying $c_e^l = \theta_0| \boldsymbol\omega^l \cdot \boldsymbol n |$ with $\theta_0$ a constant independent of $e$ and $h$.
 The stabilized discrete-ordinate DG method is to find $\boldsymbol u_h: = \{ u_h^l \} \in \boldsymbol V_h$ such that
\begin{equation}\label{stabilized}
a_h^s(\boldsymbol u_h,\boldsymbol v_h) = F(\boldsymbol v_h),\quad \boldsymbol v_h \in \boldsymbol V_h,
\end{equation}
where
\begin{equation}\label{ashswor}
a_h^s(\boldsymbol u_h,\boldsymbol v_h) = \sum\limits_{l = 1}^L w_l( a_h^{(l)}(u_h^l,v_h^l) + b_{hs}^{(l)}(u_h^l,v_h^l) )
\end{equation}
and
\begin{align}\label{bhs}
  b_{hs}^{(l)}(u_h^l,v_h^l)
 & = b_h^{(l)}(u_h^l,v_h^l) + \sum\limits_{e \in \mathcal{E}_h^0} \int_e c_e^l[\![ u_h^l ]\!] \cdot [\![ v_h^l ]\!]  {\rm d}s \nonumber \\
 & = \sum\limits_{e \not\subset \Gamma_ - } \int_e \{\!\!\{ \boldsymbol\omega ^lu_h^l \}\!\!\} \cdot [\![ v_h^l
 ]\!]  {\rm d}s + \sum\limits_{e \in \mathcal{E}_h^0} \int_e c_e^l[\![ u_h^l ]\!] \cdot [\![ v_h^l ]\!]  {\rm d}s.
\end{align}

\begin{remark}\label{rem:Vh}
The sparse grid discrete-ordinate DG method is obtained by replacing the DG space $V_h$ in \eqref{DGsingle} with the sparse DG space $\widehat{V}_h^k:=\widehat{\boldsymbol V}_N^k \subset V_h$.
\end{remark}

\section{Error analysis}

\subsection{Error estimate of the sparse projection operator}\label{subproj}

We define the broken $H^s$ Sobolev norm on $\Omega_N$ by
\[\| v \|_{H^s(\Omega_N)}^2 = \sum\limits_{\boldsymbol 0 \le \boldsymbol j \le 2^{\boldsymbol N - \boldsymbol 1} - \boldsymbol 1} \| v \|_{H^s(I_{\boldsymbol N}^{\boldsymbol j})}^2 .\]
For any nonnegative integer $m$ and the multi-index $\alpha  = \{ i_1,i_2, \cdots ,i_r \} \subset \{ 1,2, \cdots ,d \}$, define
\[| v |_{H^{m,\alpha }(\Omega )} = \Big\| \Big( \frac{\partial ^m}{\partial x_{i_1}^m} \cdots \frac{\partial ^m}{\partial x_{i_r}^m} \Big)v \Big\|_{L^2(\Omega )}\]
and
\[| v |_{\mathcal{H}^{q + 1}(\Omega )} = \mathop {\max }\limits_{1 \le r \le d} \Big( \mathop {\max }\limits_{\alpha  \in \{ 1,2, \cdots ,d \},| \alpha  | = r} | v |_{H^{q + 1,\alpha }(\Omega )} \Big),\]
which is the norm for the mixed derivative of $v$ of at most degree $q + 1$ in each direction.

In the following, we denote by $\boldsymbol P$ the sparse projection operator to be the $L^2$ projection onto $\widehat{\boldsymbol V}_N^k$.
\begin{lemma}\label{lem:sparseProj}
  Let $\boldsymbol P$ be the sparse projector, $k \ge 1$, $N \ge 1$ and $d \ge 2$. Then for $v \in \mathcal{H}^{p + 1}(\Omega )$ there hold
  \[| \boldsymbol Pv - v |_{L^2(\Omega_N)} \lesssim | \log_2h |^dh^{k + 1}| v |_{\mathcal{H}^{k + 1}(\Omega )},\]
  \[| \boldsymbol Pv - v |_{H^1(\Omega_N)} \lesssim h^k| v |_{\mathcal{H}^{k + 1}(\Omega )},\]
  and
  \[\Big( \sum\limits_{K \in \mathcal{T}_h} \| \boldsymbol Pv - v \|_{0,\partial K}^2 \Big)^{1/2}
  \lesssim | \log_2h |^dh^{k + 1/2}|v |_{\mathcal{H}^{k + 1}(\Omega )} .  \]
\end{lemma}
\begin{proof}
  It follows from \cite{Schwab-Suli-Todor-2008,Wang-Tang-Guo-2016,Guo-Cheng-2016} that for any $v \in \mathcal{H}^{p + 1}(\Omega )$ and $1\le q\le \min\{p,k\}$, there holds
  \[
  | \boldsymbol Pv - v |_{H^s(\Omega_N)}   \lesssim
  \begin{cases}
  N^d2^{ - N(q + 1)}| v |_{\mathcal{H}^{q + 1}(\Omega )},\quad & s = 0, \\
  2^{ - Nq}| v |_{\mathcal{H}^{q + 1}(\Omega )},\quad & s = 1.
\end{cases}
\]
Noting that $h =h_N= 2^{ - N}$, we have
\[| \boldsymbol Pv - v |_{L^2(\Omega_N)} \lesssim | \log_2h |^dh^{k + 1}| v |_{\mathcal{H}^{k + 1}(\Omega )}\]
and
\[| \boldsymbol Pv - v |_{H^1(\Omega_N)} \lesssim h^k| v |_{\mathcal{H}^{k + 1}(\Omega )}.\]
Recalling the trace inequality (cf. \cite{Brenner2008})
\begin{equation}\label{traceinq}
\| \phi  \|_{0,\partial K}^2 \lesssim h_K^{ - 1}\| \phi  \|_{0,K}^2 + h_K| \phi  |_{1,K}^2,
\end{equation}
where $K\in \mathcal{T}_h$ with diameter $h_K$, we then have
\begin{align*}
  \Big( \sum\limits_{K \in \mathcal{T}_h} \| \boldsymbol Pv - v \|_{0,\partial K}^2 \Big)^{1/2}
  &\lesssim \Big( \sum\limits_{K \in \mathcal{T}_h} h_K^{ - 1}\| \boldsymbol Pv - v \|_{0,K}^2 + h_K| \boldsymbol Pv - v |_{1,K}^2 \Big)^{1/2} \\
  & \lesssim ( | \log_2h |^{2d} + 1 )^{1/2}h^{k + 1/2}| v |_{\mathcal{H}^{k + 1}(\Omega )}  \\
  & \lesssim | \log_2h |^dh^{k + 1/2}|v |_{\mathcal{H}^{k + 1}(\Omega )} .
\end{align*}
This completes the proof.
\end{proof}

\subsection{Error analysis of the sparse grid discrete-ordinate DG method}

Using the similar arguments in \cite{Brezzi-Marini-Suli-2004,Han-Huang-Eichholz-2010}, one can deduce the following stability result whose proof is omitted for simplicity.

\begin{lemma}\label{lem:stability}
Let
\[|\!|\!|\boldsymbol v_h|\!|\!| = \Big[ \sum\limits_{l = 1}^L w_l\Big( \sum\limits_{e \in \mathcal{E}_h} \int_e c_e^l  | [\![ v_h^l]\!] |^2{\rm d}s + \int_D (v_h^l)^2 {\rm d}x \Big)  \Big]^{1/2}.\]
Under the assumption \eqref{assm}, there holds
\[a_h^s(\boldsymbol v_h,\boldsymbol v_h) \gtrsim |\!|\!|\boldsymbol v_h|\!|\!|^2,\quad \boldsymbol v_h \in \boldsymbol V_h.\]
\end{lemma}

We denote the solutions of the original problem \eqref{RTE2}, the semi-discrete problem \eqref{semibs} and the stabilized discrete-ordinate DG method \eqref{stabilized} by $\{ u(\boldsymbol x,\boldsymbol\omega ^l) \}$, $\boldsymbol u = \{ u^l(\boldsymbol x) \}$ and $\boldsymbol u_h = \{ u_h^l\} $, respectively. The error is decomposed as
\begin{equation}\label{errdecomp}
\{ u(\boldsymbol x,\boldsymbol\omega ^l) \} - \boldsymbol u_h = ( \{ u(\boldsymbol x,\boldsymbol\omega ^l) \} - \{ u^l(\boldsymbol x) \} ) + ( \{ u^l(\boldsymbol x) \} - \{ u_h^l(\boldsymbol x) \} ),
\end{equation}
and measured by
\begin{equation}\label{L2Sn}
\| u - u_h \|_h = \Big( \sum\limits_{l = 1}^L w_l \| u( \cdot ,\boldsymbol\omega ^l) - u_h^l \|_{0,D}^2\Big)^{1/2}.
\end{equation}

\begin{theorem}\label{thrm:dodg}
Let $n$ be the degree of precision of the numerical quadrature and $d\ge 2$. Then under the assumption \eqref{assm}, for the sparse grid discrete-ordinate DG method, we have
\begin{align*}
  \| u - u_h \|_h &\lesssim c(\theta_0)| \log_2h |^dh^{k + 1/2}\Big( \sum\limits_{l = 1}^L w_l| u^l |_{\mathcal{H}^{k + 1}(\Omega )}^2  \Big)^{1/2}  \\
  &\quad + c(r',g)n^{ - r'}\Big( \int_D \| u(\boldsymbol x, \cdot ) \|_{r',S^2}^2 {\rm d}x \Big)^{1/2},
\end{align*}
where, $h=2^{-N}$, $u^l$ is the solution of \eqref{semibs}, $c(\theta_0) = \theta_0^{ - 1/2} + \theta_0^{1/2}$ and $c(r',g)$ is defined in \eqref{crg}.
\end{theorem}
\begin{proof}
For the first part in \eqref{errdecomp}, let
\[\varepsilon ^l(\boldsymbol x): = u(\boldsymbol x,\boldsymbol\omega ^l) - u^l(\boldsymbol x),\quad 1 \le l \le L.\]
In view of the equation (4.19) in \cite{Han-Huang-Eichholz-2010}, one has
\[\sum\limits_{l = 1}^L w_l\int_D (\varepsilon ^l)^2 {\rm d}x  \lesssim c(r',g)^2n^{ - 2r'}\int_D \| u(\boldsymbol x, \cdot ) \|_{r',S^2}^2 {\rm d}x,\]
where
\begin{equation}\label{crg}
c(r',g): = c(r')\mathop {\sup }\limits_{\boldsymbol x \in D, \boldsymbol\omega  \in S^2} \|g(\boldsymbol x,\boldsymbol\omega  \cdot ) \|_{r',S^2}
\end{equation}
and $c(r')$ is a positive constant depending only on $r'$.

For the second part, let
\[u^l - u_h^l = (u^l - P_h^ku^l) - (u_h^l - P_h^ku^l) = : \eta ^l + \delta ^l,\]
where $P_h^k$ is the $L^2$-projection onto the sparse DG space $\widehat{V}_h^k$ (cf. Remark \ref{rem:Vh}). Similarly, we denote $\boldsymbol\eta  = \{ \eta ^l \}$ and $\boldsymbol\delta  = \{ \delta ^l \}$.  When $\boldsymbol u_h$ is replaced by the exact solution $\boldsymbol u$ of the semi-discrete problem, the weak continuity in Lemma \ref{dgforweak} yields
\[\int_e c_e^l[\![ u^l ]\!][\![ v_h^l ]\!] {\rm d}s = 0,\quad e \in \mathcal{E}_h^0.\]
From \eqref{bhs} we have $b_{hs}^{(l)}(u^l,v_h^l) = b_h^{(l)}(u^l,v_h^l)$, and hence
$a_h^s(\boldsymbol u,\boldsymbol v_h) = a_h(\boldsymbol u,\boldsymbol v_h)$,
which yields the following Galerkin orthogonality
\begin{equation}\label{GalerkinOrtho}
a_h^s(\boldsymbol u - \boldsymbol u_h,\boldsymbol v_h) = a_h(\boldsymbol u - \boldsymbol u_h,\boldsymbol v_h) = 0,\quad \boldsymbol v_h \in \boldsymbol V_h.
\end{equation}
According to the stability estimate in Lemma \ref{lem:stability}, we have
\begin{equation}\label{errfram}
  |\!|\!|\boldsymbol\delta|\!|\!|^2 \lesssim a_h^s(\boldsymbol\delta,\boldsymbol\delta ) = a_h^s(\boldsymbol u_h - P_h^k\boldsymbol u,\boldsymbol\delta )
  = a_h^s(\boldsymbol u - P_h^k\boldsymbol u,\boldsymbol\delta ) = a_h^s(\boldsymbol\eta ,\boldsymbol\delta ).
\end{equation}

We now estimate the right-hand side of \eqref{errfram}. Let
\[a_h^{(l)}(u^l,v_h) = {\rm I}_1^l(u^l,v_h) - {\rm I}_2^l(u^l,v_h),\]
where
\begin{align*}
  {\rm I}_1^l(u^l,v_h) & = \sum\limits_{K \in \mathcal{T}_h} \int_K \left(  - u^l(\boldsymbol\omega ^l \cdot \nabla v_h) + \sigma_tu^lv_h \right){\rm d}x , \\
  {\rm I}_2^l(u^l,v_h) & = \int_D \sigma_s\sum\limits_{i = 1}^L w_ig( \cdot ,\boldsymbol\omega ^l \cdot \boldsymbol\omega ^i)u^iv_h{\rm d}x.
\end{align*}
For the first term ${\rm I}_1^l$, noting that $\boldsymbol\omega ^l \cdot \nabla \delta^l|_K \in \mathbb{Q}_k(K)$, by the definition of the projector $P_h^k$,
\[\int_K \eta^l(\boldsymbol\omega ^l \cdot \nabla \delta^l) {\rm d}x = \int_K (u^l - P_h^ku^l)(\boldsymbol\omega ^l \cdot \nabla \delta^l) {\rm d}x =0,\]
which gives
\[
  |{\rm I}_1^l(\eta ^l,\delta ^l)|
  =  \Big|\sum\limits_{K \in \mathcal{T}_h} \int_K \left(  -\eta^l(\boldsymbol\omega ^l \cdot \nabla \delta^l) + \sigma_t\eta^l\delta ^l \right){\rm d}x\Big|
  \lesssim \sum\limits_{K \in \mathcal{T}_h} \| \eta ^l \|_{0,K}\| \delta ^l \|_{0,K}.
\]
The Cauchy-Schwarz inequality yields
\[
 \Big| \sum\limits_{l = 1}^L w_l{\rm I}_1^l(\eta ^l,\delta ^l)\Big|
   = \Big( \sum\limits_{l = 1}^L w_l\| \eta ^l \|_{0,D}^2 \Big)^{1/2}\Big( \sum\limits_{l = 1}^L w_l\| \delta ^l \|_{0,D}^2  \Big)^{1/2}.
\]
For the second one, using Lemma 4.3 in \cite{Han-Huang-Eichholz-2010}, we obtain
\begin{align*}
  \Big|\sum\limits_{l = 1}^L {w_l{\rm I}_2^l(\eta ^l,\delta ^l)} \Big|
  & \le \Big( \sum\limits_{l = 1}^L w_l\int_D m\sigma_s(\eta ^l)^2 {\rm d}x \Big)^{1/2}\Big( \sum\limits_{l = 1}^L w_l\int_D m\sigma_s(\delta ^l)^2 {\rm d}x\Big)^{1/2} \\
  & \lesssim \Big( \sum\limits_{l = 1}^L w_l\| \eta ^l \|_{0,D}^2  \Big)^{1/2}\Big( \sum\limits_{l = 1}^L w_l\| \delta ^l \|_{0,D}^2\Big)^{1/2},
\end{align*}
where $m = m(\boldsymbol x)$ is given in \eqref{mx}.

It remains to consider $b_{hs}^{(l)}(\eta ^l,\delta ^l)$. From (49) in \cite{Brezzi-Marini-Suli-2004} we have
\[|b_{hs}^{(l)}(\eta ^l,\delta ^l)| \le \sum\limits_{e \in \mathcal{E}_h} \Big( \frac{1}{\theta_0}\| c_e^{1/2}\{\!\!\{ \eta ^l\}\!\!\} \|_{0,e} + \| c_e^{1/2}[\![ \eta ^l ]\!] \|_{0,e} \Big) \| c_e^{1/2}[\![ \delta ^l ]\!] \|_{0,e}.\]
According to the choice of $c_e^l$, the Cauchy-Schwarz inequality gives
\[\Big|\sum\limits_{l = 1}^L {w_lb_{hs}^{(l)}(\eta ^l,\delta ^l)} \Big| \lesssim c(\theta_0)\Big( \sum\limits_{l = 1}^L w_l\sum\limits_{K \in \mathcal{T}_h} \| \eta ^l \|_{0,\partial K}^2 \Big)^{1/2}\Big( \sum\limits_{l = 1}^L w_l\sum\limits_{e \in \mathcal{E}_h} \| c_e^{1/2}[\![ \delta ^l ]\!]\|_{0,e}^2\Big)^{1/2},\]
where $c(\theta_0) = \theta_0^{ - 1/2} + \theta_0^{1/2}$ is a constant,
and hence
\[
  a_h^s(\boldsymbol\eta ,\boldsymbol\delta )
  \lesssim c(\theta_0)\Big[ \sum\limits_{l = 1}^L w_l\Big( \| \eta ^l \|_{0,D}^2 + \sum\limits_{K \in \mathcal{T}_h} \| \eta ^l \|_{0,\partial K}^2\Big)\Big]^{1/2}|\!|\!|\boldsymbol\delta|\!|\!|.
\]
This combined with \eqref{errfram} yields
\[|\!|\!|\boldsymbol\delta|\!|\!| \lesssim c(\theta_0)\Big[ \sum\limits_{l = 1}^L w_l\Big( \| \eta ^l \|_{0,D}^2 + \sum\limits_{K \in \mathcal{T}_h} \| \eta ^l \|_{0,\partial K}^2  \Big) \Big]^{1/2},\]
and with the error estimates of the sparse projection in Lemma \ref{lem:sparseProj} leads to the desired result.
\end{proof}

\section{Numerical results} \label{sec:numer}

In this section, we shall provide a series of numerical examples for solving the RTE \eqref{RTE2}-\eqref{inflowbd} to illustrate the performance of the proposed sparse grid discrete coordinate DG method.

\subsection{The linear system from the discrete problem}

The $S_n$ method has $n(n+2)$ directions with $n$ an even natural number. The discrete-ordinate sets satisfying the required moment equations to fourteen digits of accuracy have been given in \cite{Balsara-2001}.  Note that only the ordinates in the first octant are given there. The remaining ordinates can be obtained by using symmetry arguments. For example, $S_2$ data is given in Tab. \ref{S2sets}.
\begin{table}[!htb]
  \centering
  \caption{Discrete-ordinate sets for $S_2$~($\boldsymbol\omega = (s_1,s_2,s_3)$)}\label{S2sets}
  \begin{tabular}{llllll}
    \hline
    $s_1$ & $s_2$ & $s_3$ & $w$ \\ \hline
    $\pm0.5773502691896257$ & $\pm0.5773502691896257$ & $\pm0.5773502691896257$ & 1.5707963267948966 \\
    \hline
  \end{tabular}
\end{table}

Relabel the sparse bases by a single index $i=1,2,\cdots,M$ and denote them by $\varphi_i$, where $M = \dim \widehat{\boldsymbol V}_N^k$. The variational problem \eqref{stabilized} can be written in matrix form
  \begin{equation}\label{uzfd}
\boldsymbol A^{(l)}\hat{\boldsymbol U}^l - \sum\limits_{i = 1}^L \boldsymbol B_i^{(l)}\hat{\boldsymbol U}^i  + \boldsymbol C^{(l)}\hat{\boldsymbol U}^l = \boldsymbol F^{(l)},\quad 1 \le l \le L,
\end{equation}
where
 \[\boldsymbol A^{(l)} = ( a_{nm}^{(l)} ),\quad \boldsymbol B_i^{(l)} = ( b_{nmi}^{(l)} ),\quad \boldsymbol C^{(l)} = ( c_{nm}^{(l)} ),\]
 \[\hat{\boldsymbol U}^l = [ \hat u_1^l,\hat u_2^l, \cdots ,\hat u_M^l ]^T,\quad \boldsymbol F^{(l)} = [ F_1^{(l)},F_2^{(l)}, \cdots ,F_M^{(l)} ]^T,\]
 and
\[a_{nm}^{(l)} = \sum\limits_{K \in \mathcal{T}_h} \int_K (  - \varphi_m(\boldsymbol\omega ^l \cdot \nabla \varphi_n) + \sigma_t\varphi_m\varphi_n ) {\rm d}x ,\]
\[b_{nmi}^{(l)} = w_i\int_D \sigma_sg( \cdot ,\boldsymbol\omega ^l \cdot {\boldsymbol\omega} ^i){\varphi_m}\varphi_n {\rm d}x,\]
\[c_{nm}^{(l)} = \sum\limits_{e\not  \subset \Gamma_ - } \int_e \{\!\!\{ \boldsymbol\omega ^l\varphi_m \}\!\!\} \cdot [\![ \varphi_n ]\!]  {\rm d}s + \sum\limits_{e \in \mathcal{E}_h^0} \int_e c_e^l[\![ \varphi_m ]\!] \cdot [\![ \varphi_n ]\!]  {\rm d}s,\]
\[F_n^{(l)} = \int_D f^l\varphi_n {\rm d}x - \sum\limits_{e \subset \Gamma_ -} \int_e \boldsymbol\omega ^l \cdot \boldsymbol n\alpha^l\varphi_n {\rm d}s.\]
The system \eqref{uzfd} can be further rewritten in block matrix form
 \[\boldsymbol D^{(l)}\hat{\boldsymbol U} = \boldsymbol F^{(l)},\quad 1 \le l \le L,\]
where
\[ \boldsymbol D^{(l)}: = [  - \boldsymbol B_1^{(l)}, \cdots , - \boldsymbol B_{l - 1}^{(l)},{\boldsymbol A^{(l)}} - \boldsymbol B_l^{(l)} + \boldsymbol C^{(l)}, \cdots , - \boldsymbol B_L^{(l)} ].\]
The final linear system is
\begin{equation}\label{DUF}
\boldsymbol D\hat{\boldsymbol U} = \boldsymbol F,
\end{equation}
where
 \[\boldsymbol D = \left[ \begin{array}{*{20}{c}}
  \boldsymbol D^{(1)} \\
   \vdots  \\
  \boldsymbol D^{(L)}
\end{array} \right],\quad \hat{\boldsymbol U} = \left[ \begin{array}{*{20}{c}}
  \hat{\boldsymbol U}^1 \\
   \vdots  \\
  {\hat{\boldsymbol U}}^L
\end{array} \right],\quad \boldsymbol F = \left[ {\begin{array}{*{20}{c}}
  \boldsymbol F^{(1)} \\
   \vdots  \\
  \boldsymbol F^{(L)}
\end{array}} \right].\]
We solve \eqref{DUF} by using the block Gauss-Seidal iteration method.

The accuracy is measured by the weighted relative error defined by
\[\| u - u_h \|_{rel} = \frac{\left( \sum\limits_{l = 1}^L \omega_l\| u( \cdot ,\boldsymbol\omega ^l) - u_h^l \|_{0,D}^2 \right)^{1/2}}{\left( \sum\limits_{l = 1}^L \omega_l\| u( \cdot ,\boldsymbol\omega ^l) \|_{0,D}^2\right)^{1/2}},\]
where $u_h^l = \sum\limits_{m = 1}^M \hat u_m^l\varphi_m$ is the numerical solution.

\subsection{Examples in three dimensions}

\begin{example}\label{ex:1}
We take $\sigma_t = 2$, $\sigma_s = 1$ and $\eta  = 0$. The domain $D$ is a unit cube. With the right-hand side function
\begin{align*}
  f(\boldsymbol x,\boldsymbol\omega ) &= \pi s_1\cos (\pi x_1)\sin (\pi x_2)\sin (\pi x_3) + \pi s_2\sin (\pi x_1)\cos (\pi x_2)\sin (\pi x_3) \\
  & \quad  + \pi s_3\sin (\pi x_1)\sin (\pi x_2)\cos (\pi x_3) + \sin (\pi x_1)\sin (\pi x_2)\sin (\pi x_3),
\end{align*}
where $\boldsymbol\omega  = (s_1,s_2,s_3)$, the exact solution is
\[u(\boldsymbol x,\boldsymbol\omega ) = \sin (\pi x_1)\sin (\pi x_2)\sin (\pi x_3).\]
\end{example}

\begin{figure}[!htb]
\centering
\includegraphics[scale=0.5]{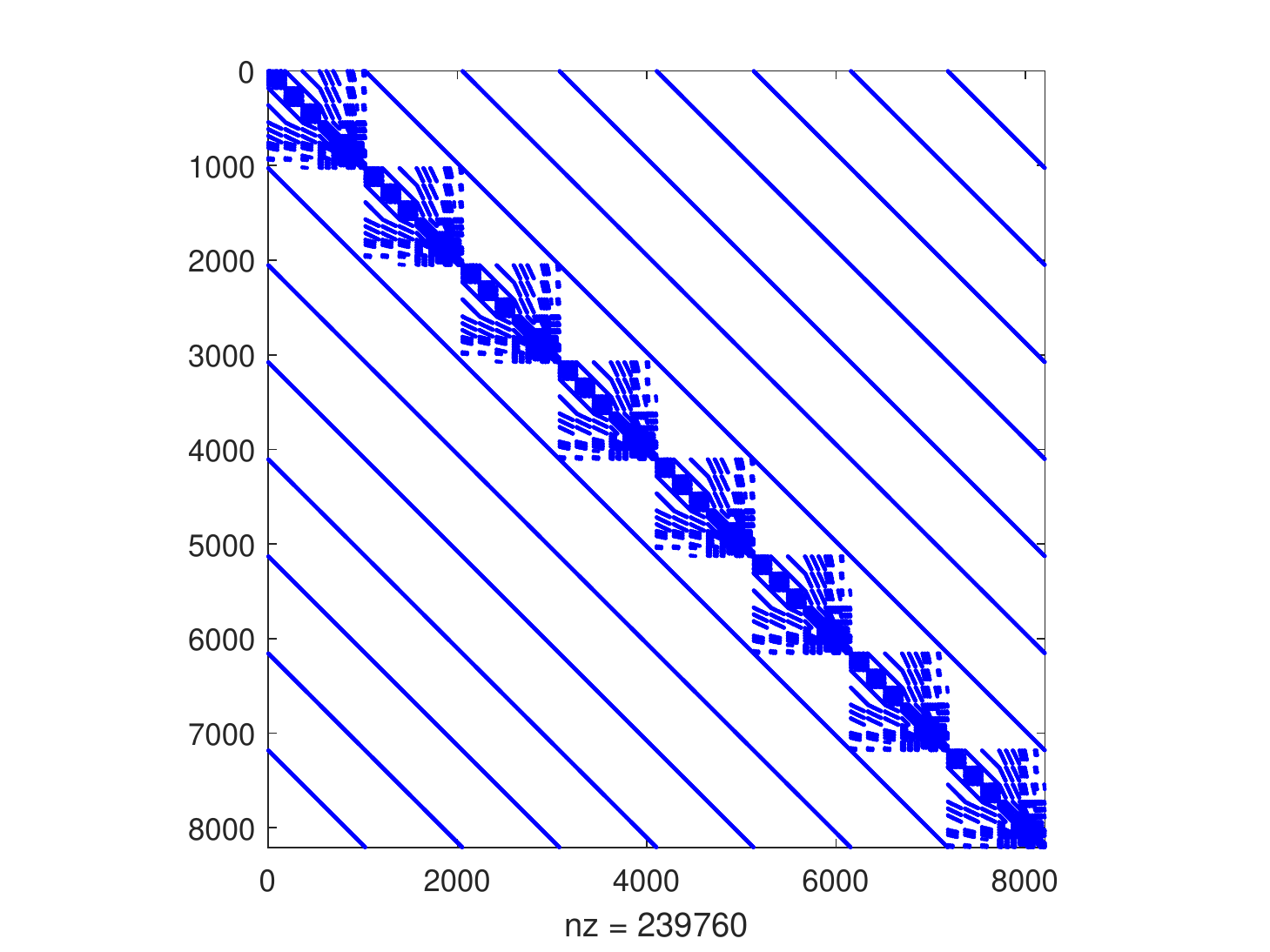}\\
\caption{Sparse pattern of the coefficient matrix for Example \ref{ex:1} ($N=3,k=2, n=2$)}\label{spar}
\end{figure}
The sparse pattern for the coefficient matrix is shown in Fig.~\ref{spar}. The total number of the entries is $8200 \times 8200 = {\text{67371264}}$ and the number of nonzero elements is ${\text{nz}} = 239760$. Thus the sparsity ratio is 99.64\%.
\begin{figure}[!htb]
  \centering
  \subfigure[Exact]{\includegraphics[scale=0.45]{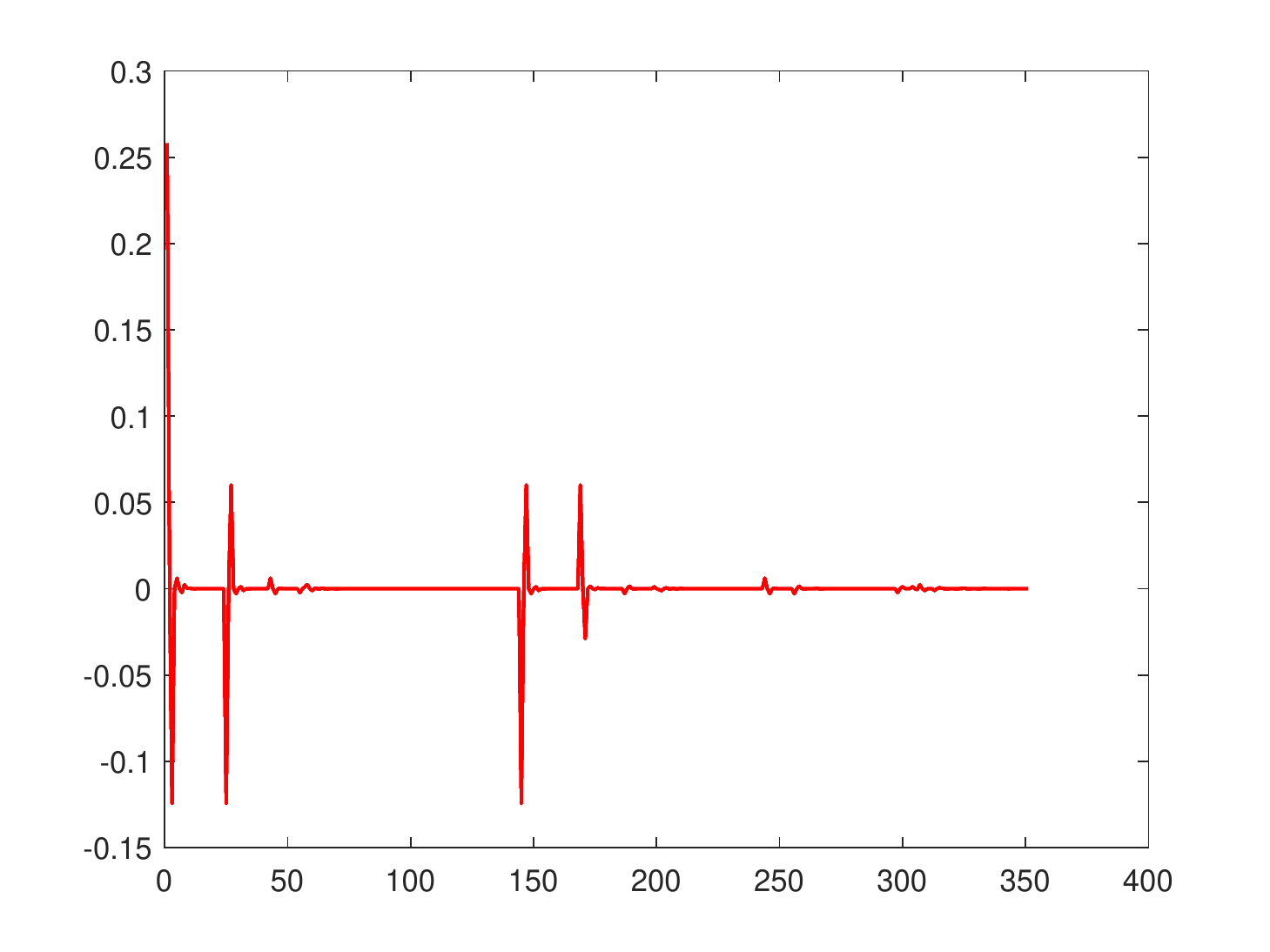}}
  \subfigure[$\theta_0 = 10$]{\includegraphics[scale=0.45]{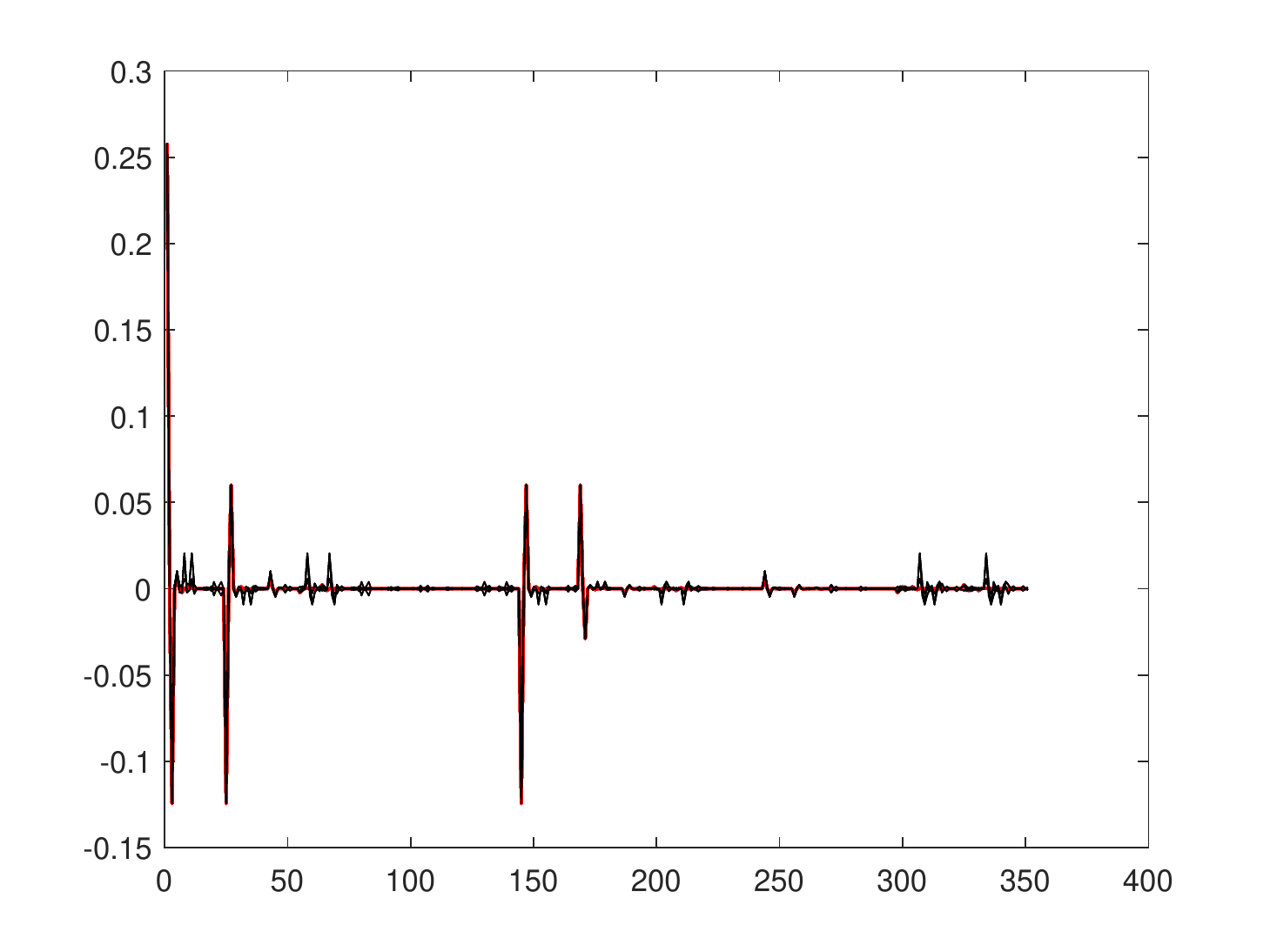}}\\
  \subfigure[$\theta_0 = 100$]{\includegraphics[scale=0.45]{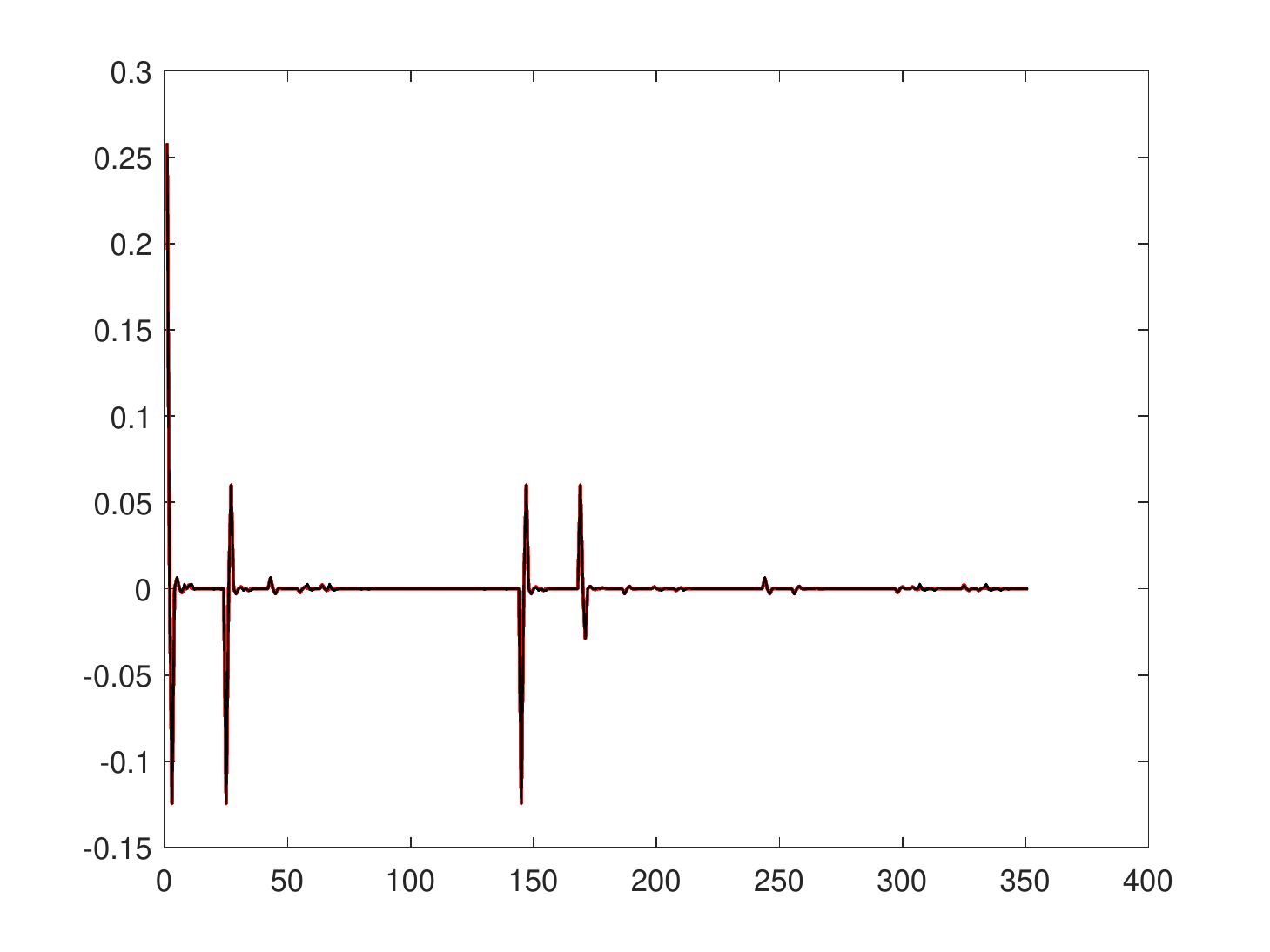}}
  \subfigure[$\theta_0 = 500$]{\includegraphics[scale=0.45]{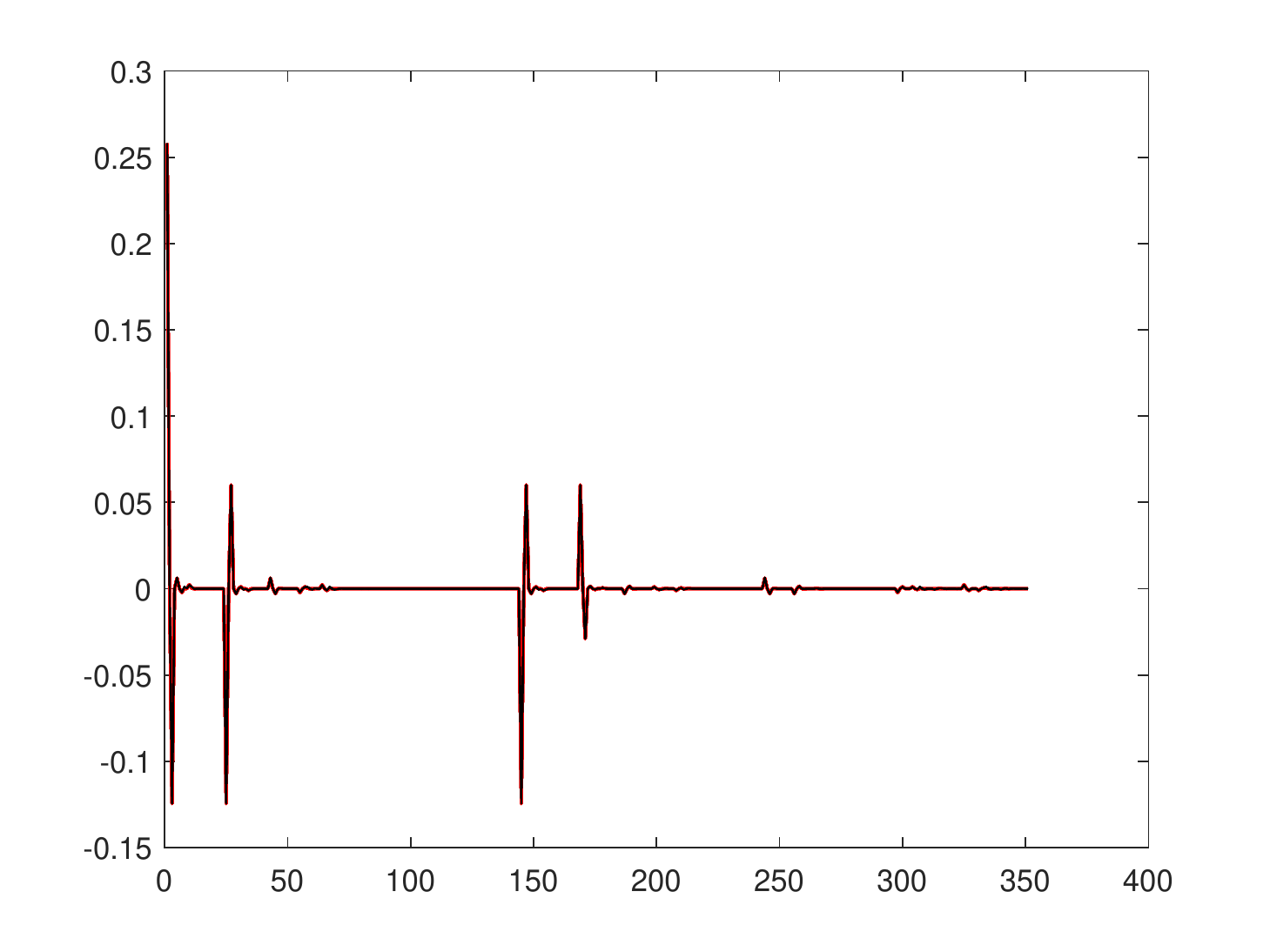}}\\
  \caption{The expansion coefficients for Example \ref{ex:1} with different stabilization parameters ~($S_2, k=2$)}\label{cofe}
\end{figure}

\begin{table}[!htb]
  \centering
  \caption{Relative errors for Example \ref{ex:1}: $k$ v.s. $S_n$ ($N = 2$)}\label{tab:ksn}
\begin{tabular}{ccccccccccccccc}
  \hline
    $n$      & 2          & 4          & 6          & 8          & 10 \\ \hline
    $k = 1$  & 1.7133e-01 & 1.7329e-01 & 1.7480e-01 & 1.7491e-01 & 1.7500e-01\\
    $k = 2$  & 8.9453e-03 & 8.3749e-03 & 8.0365e-03 & 8.0532e-03 & 8.0755e-03\\
    \hline
\end{tabular}
\end{table}
\begin{table}[!htb]
  \centering
  \caption{Relative errors for Example \ref{ex:1}: $N$ v.s. $S_n$ ($k = 2$)}\label{Nsn}
\begin{tabular}{ccccccccccccccc}
  \hline
    $n$    & 2          & 4          & 6          & 8          & 10 \\ \hline
    $N = 2$  & 8.9453e-03 & 8.3749e-03 & 8.0365e-03 & 8.0532e-03 & 8.0755e-03\\
    $N = 3$  & 2.2512e-03 & 2.1269e-03 & 2.0150e-03 & 2.0138e-03 & 2.0136e-03 \\
    \hline
\end{tabular}
\end{table}
Fig.~\ref{cofe} displays the expansion coefficients which coincide in each angular direction since the true solution is independent of the angular variable $\boldsymbol\omega$. We observe a better result for bigger stabilization parameter $\theta_0$. In the following we always choose $\theta_0 = 10^{N + k}$ due to its good performance in different cases. For the isotropic case $\eta =0$, $S_2$ method is enough to resolve the solution accurately in angle as indicated by the numerical results in Tabs.~\ref{tab:ksn} and \ref{Nsn}. For the given example, the error is then dominated by the spatial problems. According to Theorem \ref{thrm:dodg}, the error bound is $\mathcal{O}(c(\theta_0)| \log_2h |^dh^{k + 1/2})$. With the choice for $\theta_0$ in this case, we have $c(\theta_0)| \log_2h |^dh^{k + 1/2} \approx \mathcal{O}(| \log_2h |^dh^k)$ and the logarithmic factor implies a slightly lower order than $k$. From Tab.~\ref{tab:S2ex1}, we see that the convergence rates for $k=1,3$ are better than $\mathcal{O}(h^{k + 1/2})$ and even the $(k+1)$-th order can be obtained for $k=3$. For $k=2,4$ the order is about $k$.

\begin{table}[!htb] 
\footnotesize
\caption{$L^2$ errors of $S_2$ method for Example \ref{ex:1}} \label{tab:S2ex1}
\centering
\begin{tabular}{ccccccccccccccccccccccccccc}
\hline
\multirow{2}{*}{$N$} &\multicolumn{2}{c}{$k=1$} &&\multicolumn{2}{c}{$k=2$} &&\multicolumn{2}{c}{$k=3$} &&\multicolumn{2}{c}{$k=4$}\\
\cline{2-3} \cline{5-6} \cline{8-9} \cline{11-12}
  & Err & rate         && Err & rate            && Err & rate         && Err & rate\\  
\hline
1 &4.8695e-01 &-       &&3.7626e-02 &-          &&3.8603e-03 &-       &&2.9324e-04 & -    \\
2 &1.7133e-01 &1.5070  &&8.9453e-03 &2.0725     &&2.1133e-04 &4.1911  &&1.6406e-05 & 4.1598\\
3 &5.6436e-02 &1.6021  &&2.2512e-03 &1.9904     &&1.2971e-05 &4.0261  &&7.8260e-07 & 4.3898\\
4 &1.6990e-02 &1.7319  &&5.6295e-04 &1.9996     &&8.2285e-07 &3.9785  && -         & -\\
\hline
\end{tabular}
\end{table}

 \begin{example}\label{ex:2}
  We take $\sigma_t = 3$ and $\sigma_s = 1$. The domain $D$ is a unit cube. The true solution is taken as
\[u(\boldsymbol x, \boldsymbol\omega)=10 \omega_3 \sin (\pi x_1) \sin (\pi x_2) \sin (\pi x_3),\]
from which we know after a direct manipulation that the right-hand side function is
\begin{align*}
f(\boldsymbol x, \boldsymbol{s})=& 10(\sigma_t -\eta \sigma_s )s_3 \sin (\pi x_1) \sin (\pi x_2) \sin (\pi x_3) \\
&+10 \pi s_3^2 \sin (\pi x_1) \sin (\pi x_2) \cos (\pi x_3)+10 \pi s_2 s_{3} \sin (\pi x_1) \cos (\pi x_2) \sin (\pi x_3) \\ &+10 \pi s_1 s_3 \cos (\pi x_1) \sin (\pi x_2) \sin (\pi x_3)
\end{align*}
where $\boldsymbol\omega = (s_1,s_2,s_3)$.
\end{example}

\begin{table}[!htb]
  \centering
  \caption{Relative errors for Example \ref{ex:2}~($S_2, \eta = 0.1$)}\label{tab:closeis}
\begin{tabular}{ccccccccccccccc}
  \hline
    $N$ & $k = 1$ & $k = 2$   & $k = 3$  & $k = 4$ \\ \hline
    1   & 2.2797e-01 & 1.6584e-02 & 1.7683e-03 & 2.1850e-04\\
    2   & 8.2048e-02 & 3.7848e-03 & 2.0045e-04 & 1.7282e-04\\
  \hline
\end{tabular}
\end{table}
  \begin{table}[!htb]
  \centering
  \caption{Relative errors for Example \ref{ex:2}~($N=1, \eta = 0.9$)}
  \label{tab:eta09}
\begin{tabular}{ccccccccccccccc}
  \hline
    $n$ & $k = 1$ & $k = 2$   & $k = 3$  & $k = 4$ \\ \hline
    2   & 6.6685e-01 & 6.6872e-01 & 6.6969e-01 & 2.0653e+01\\
    4   & 6.4693e-01 & 1.4665e+01 & 1.5884e+01 & 1.6226e+01\\
    6   & 1.1411e+00 & 1.2068e+00 & 1.2539e+00 & 1.2606e+00\\
    8   & 1.1511e-01 & 1.3098e-01 & 1.3206e-01 & 1.3212e-01\\
    10  & 7.3980e-02 & 7.8657e-02 & 7.9128e-02 & 7.9144e-02 \\
    12  & 4.3524e-02 & 3.2606e-02 & 3.2684e-02 & 3.2690e-02 \\
  \hline
\end{tabular}
\end{table}

We observe from Tab.~\ref{tab:closeis} that $S_2$ method is accurate enough for the anisotropy factor close to isotropic cases. However, for strong forward scattering of $\eta = 0.9$, it does not give a satisfactory result. We have to choose a larger $n$ to get an improved result, which, however, is not expected in real applications since $S_n$ method has $n(n+2)$ angular directions and hence $n(n+2)$ coupled spatial problems. In this case, some models have been developed to
approximate the integral operator (cf. \cite{Han-Eichholz-Wang-2012,Sheng-Han-2013,Zheng-Han-2011}).
Another approach is to combine the sparse grid technique with the spherical harmonic method.

\begin{figure}[!tbh]
\centering
\includegraphics[scale=0.5]{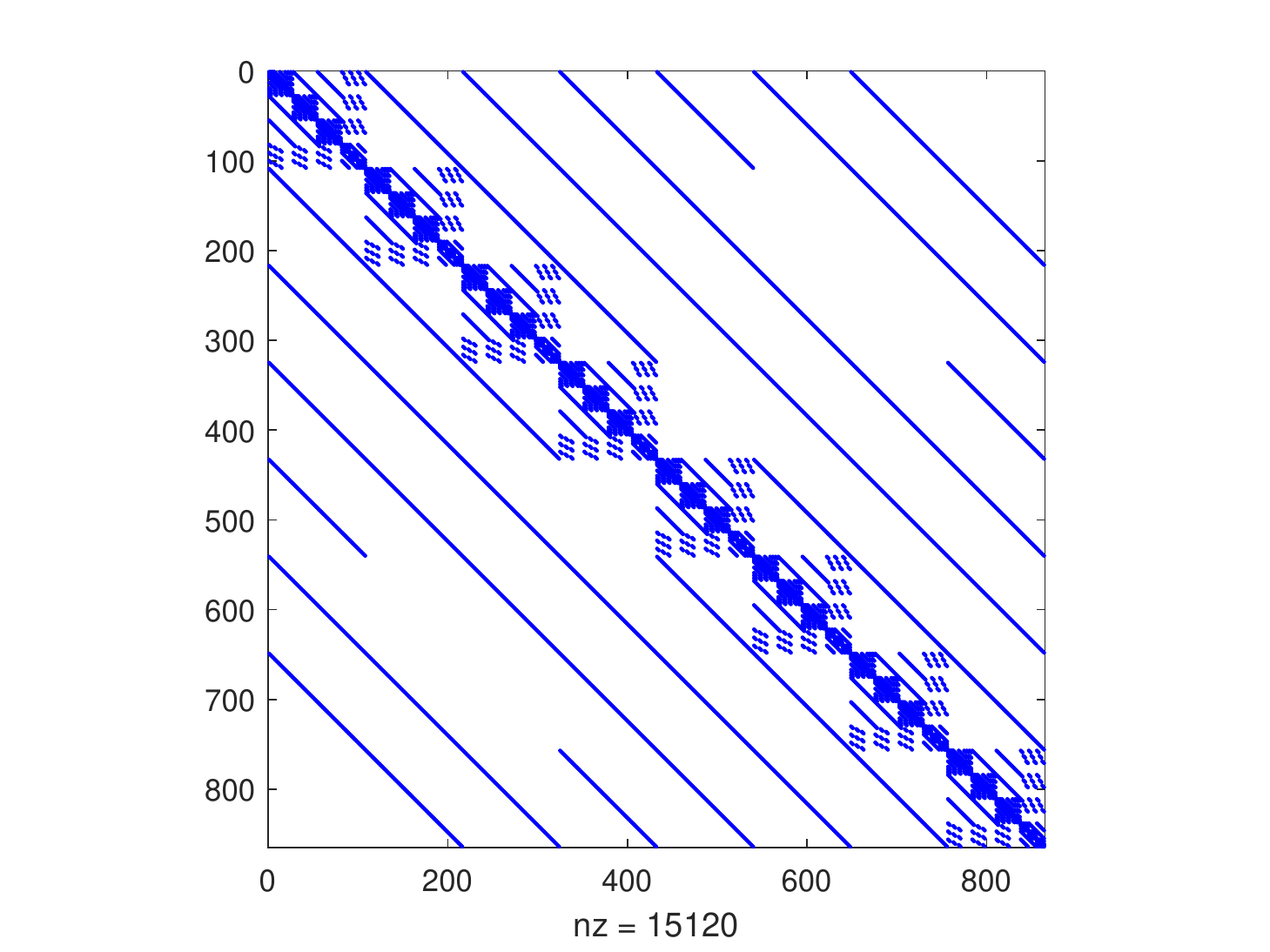}\\
\caption{Sparse pattern of the coefficient matrix for Example \ref{exsam} ($N=1,k=2, n=2$)}\label{SAM0}
\end{figure}

\begin{example}\label{exsam}
This example is taken from the reference \cite{Liu-1994}, where the Henyey-Greenstein function is replaced by the simplified approximate Mie (SAM):
\[
g(t) = K_S(1+t)^{n_p},~~t \in [ - 1,1],
\]
where $n_p = \frac{2\eta}{1-\eta}$ is the anisotropic index and $K_S = \frac{1}{2\pi}\frac{n_p+1}{2^{n_p+1}}$ is the
normalization factor. The geometric  parameters and the true solution are the same as Example \ref{ex:2}.
\end{example}

For $S_2$ method with $N=1$ and $k = 2$, the sparse pattern for the coefficient matrix is shown in Fig.~\ref{SAM0}. We also display the numerical and exact coefficients and $L^2$ projections at $z=0$ associated with the first angular direction in Fig.~\ref{SAM1}. We repeat the test for highly forward-peaked scattering with $\eta =0.9$.
From Tab.~\ref{tab:SAMeta09} we observe a relatively smaller errors than that from Tab.~\ref{tab:eta09}, but the convergence behaviours are the same since the errors do not decrease significantly with the increase of $k$ and $n$.

\begin{figure}[!tbh]
\centering
\includegraphics[scale=0.7]{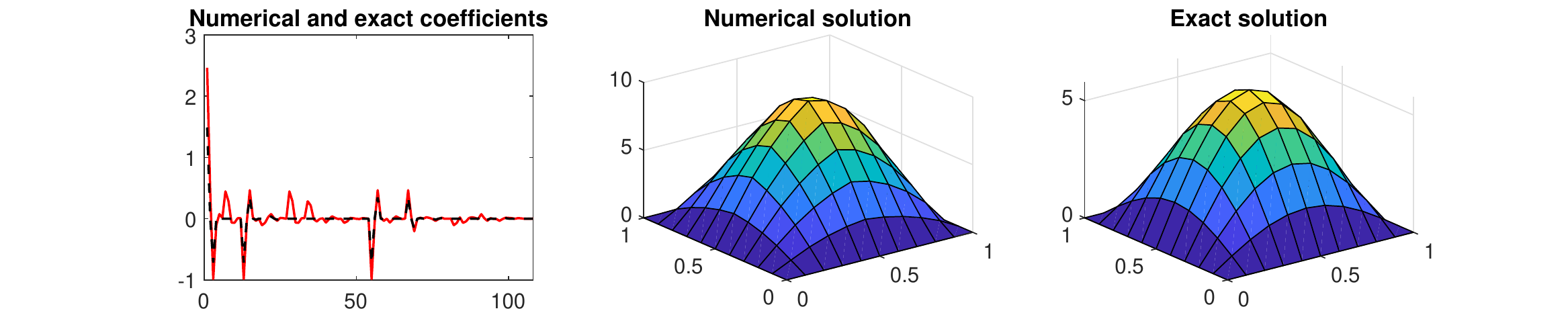}\\
\caption{Numerical and exact coefficients and $L^2$ projections for Example \ref{exsam} ($N=1,k=2, n=2$)}\label{SAM1}
\end{figure}

 \begin{table}[!htb]
  \centering
  \caption{Relative errors for Example \ref{exsam}~($N=1, \eta = 0.9$)}
  \label{tab:SAMeta09}
\begin{tabular}{ccccccccccccccc}
  \hline
    $n$ & $k = 1$ & $k = 2$   & $k = 3$  & $k = 4$ \\ \hline
    2   & 6.0818e-01 & 7.2904e-01 & 7.3901e-01 & 7.3967e-01\\
    4   & 3.7295e-01 & 4.0517e-02 & 3.1238e-02 & 3.1143e-02\\
    6   & 3.7622e-01 & 3.7490e-02 & 2.7350e-02 & 2.7243e-02\\
    8   & 3.7823e-01 & 2.8034e-02 & 1.0114e-02 & 9.7213e-03\\
    10  & 3.7872e-01 & 2.6477e-02 & 3.5129e-03 & 2.0719e-03 \\
    12  & 3.7875e-01 & 2.6452e-02 & 3.2775e-03 & 1.6395e-03 \\
  \hline
\end{tabular}
\end{table}

\subsection{Flux distributions in two and three dimensions}

We now investigate the impact of the source term on the flux distributions. The isotropic photon flux is defined by
\[q(\boldsymbol{x}) = \frac{1}{4\pi}\int_{S^2} u(\boldsymbol{x},\hat{\boldsymbol\omega} ) {\rm d}\sigma (\hat{\boldsymbol\omega} ).\]
For simplicity, vacuum boundary conditions are applied on all the boundaries. We always consider the isotropic scattering, and take $N=k=2$ for the spatial discretization. The examples in this subsection are taken from the reference \cite{Roberts-2010}.

\begin{example}\label{ex:source3D}
  This problem is defined on a unit cube with vacuum boundaries. The first 0.2 by 0.2 by 0.2 region $R$ contains a uniform isotropic source. For simplicity, we consider the following right-hand side function:
 \[f(\boldsymbol{x},\boldsymbol\omega) = f(\boldsymbol{x}) =
 \begin{cases}
 1, \quad \boldsymbol{x}\in R = [0,0.2]^3, \\
 0, \quad \boldsymbol{x}\in D\backslash R.
 \end{cases}\]
   The entire box is of uniform composition with the following data: $\sigma_t = 1$ and $\sigma_s = 0.4$.
\end{example}

For $z=0.1$ fixed, the contour plot of the flux distributions with varying orders of the discrete ordinates is displayed in Fig.~\ref{Fig:source3D}. We can see clearly that the contour map shows rays emanating from the source.
\begin{figure}[!tbh]
  \centering
  \subfigure[$S_2$]{\includegraphics[scale=0.35]{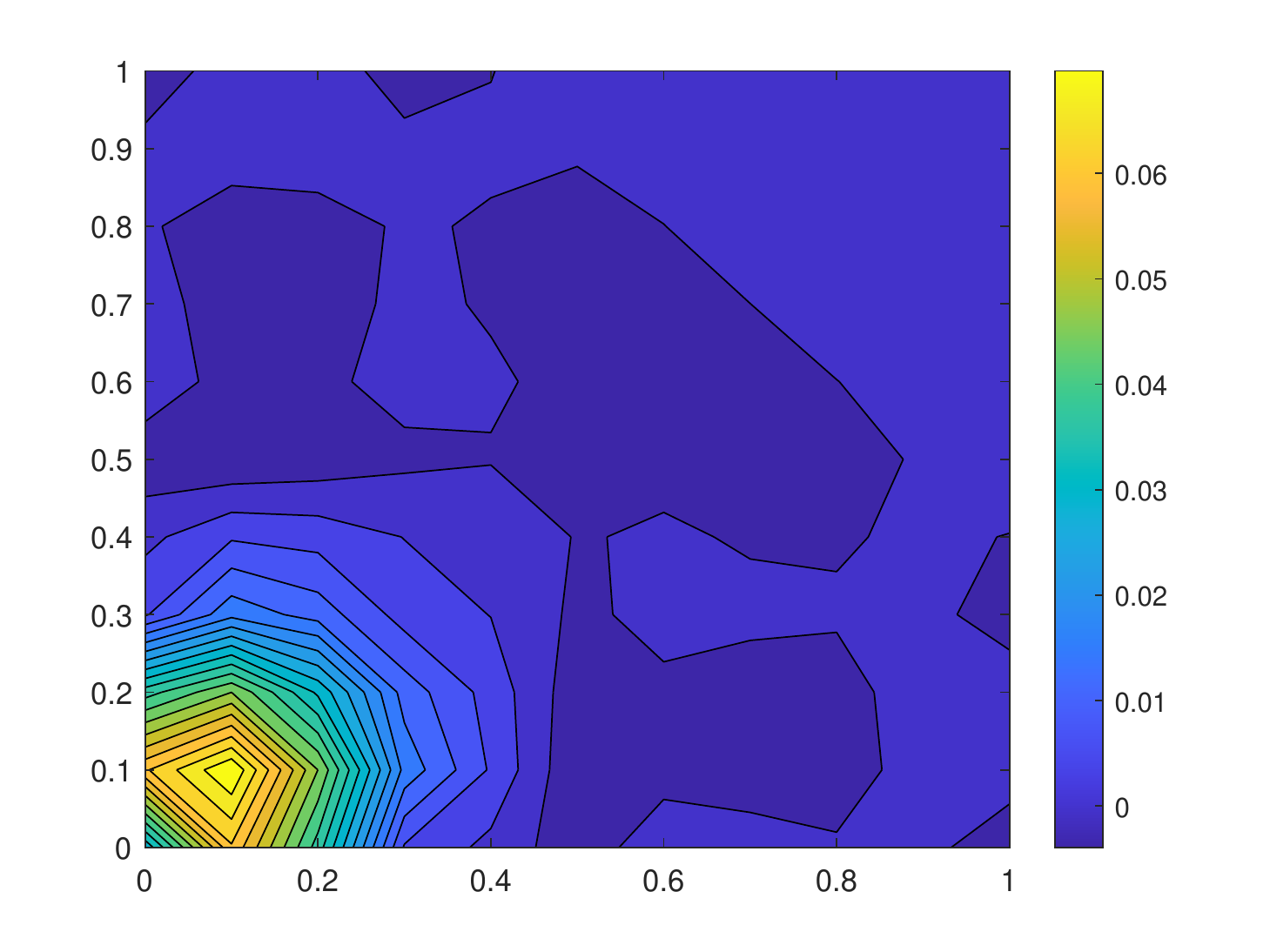}}
  \subfigure[$S_4$]{\includegraphics[scale=0.35]{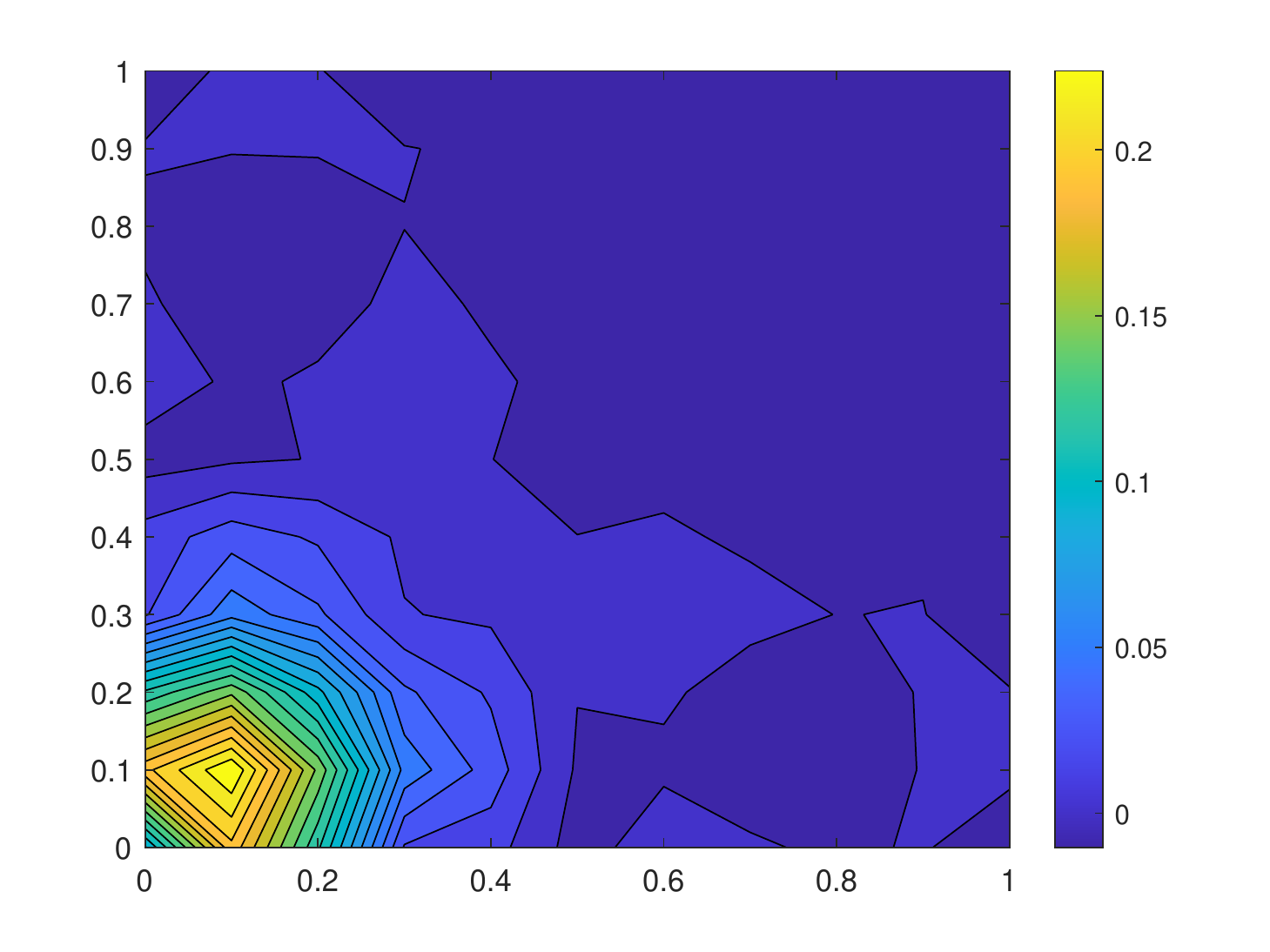}}
  \subfigure[$S_6$]{\includegraphics[scale=0.35]{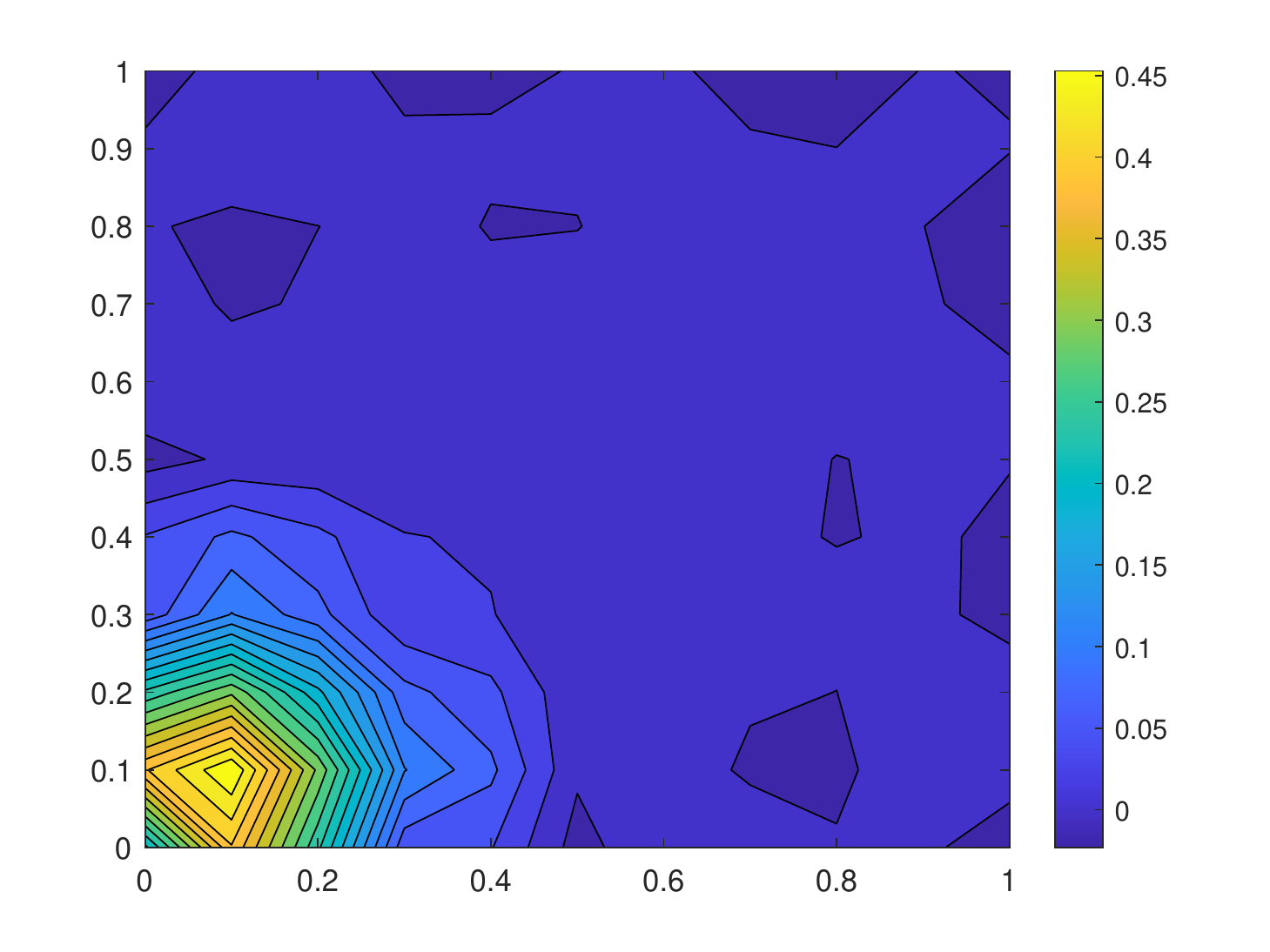}}\\
  \caption{Contour map of the photon flux distributions for Example \ref{ex:source3D}} \label{Fig:source3D}
\end{figure}

We now perform the test on the problem in $(x, y)$-geometry. In this case, all coefficients, the boundary data and the solution of \eqref{RTE2}-\eqref{inflowbd} are independent of the space variable $x_3 = z$.

\begin{example}\label{ex:source}
  This problem is defined on a unit square with vacuum boundaries. The first 0.2 by 0.2 region localized in the lower left corner contains a uniform isotropic source. The entire box is of uniform composition with the following data: $\sigma_t = 1$ and $\sigma_s = 0.4$.
\end{example}

For this example, we only consider the $S_4$ method. The contour plot of the flux distributions is displayed in Fig.~\ref{Fig:source}~(a). Numerical results of other cases are listed in Fig.~\ref{Fig:source}~(b)-(f) by changing the positions or increasing the numbers of the isotropic sources. In all cases, we again observe the rays emanating from the sources.

\begin{figure}[!tbh]
  \centering
  \subfigure[]{\includegraphics[scale=0.35]{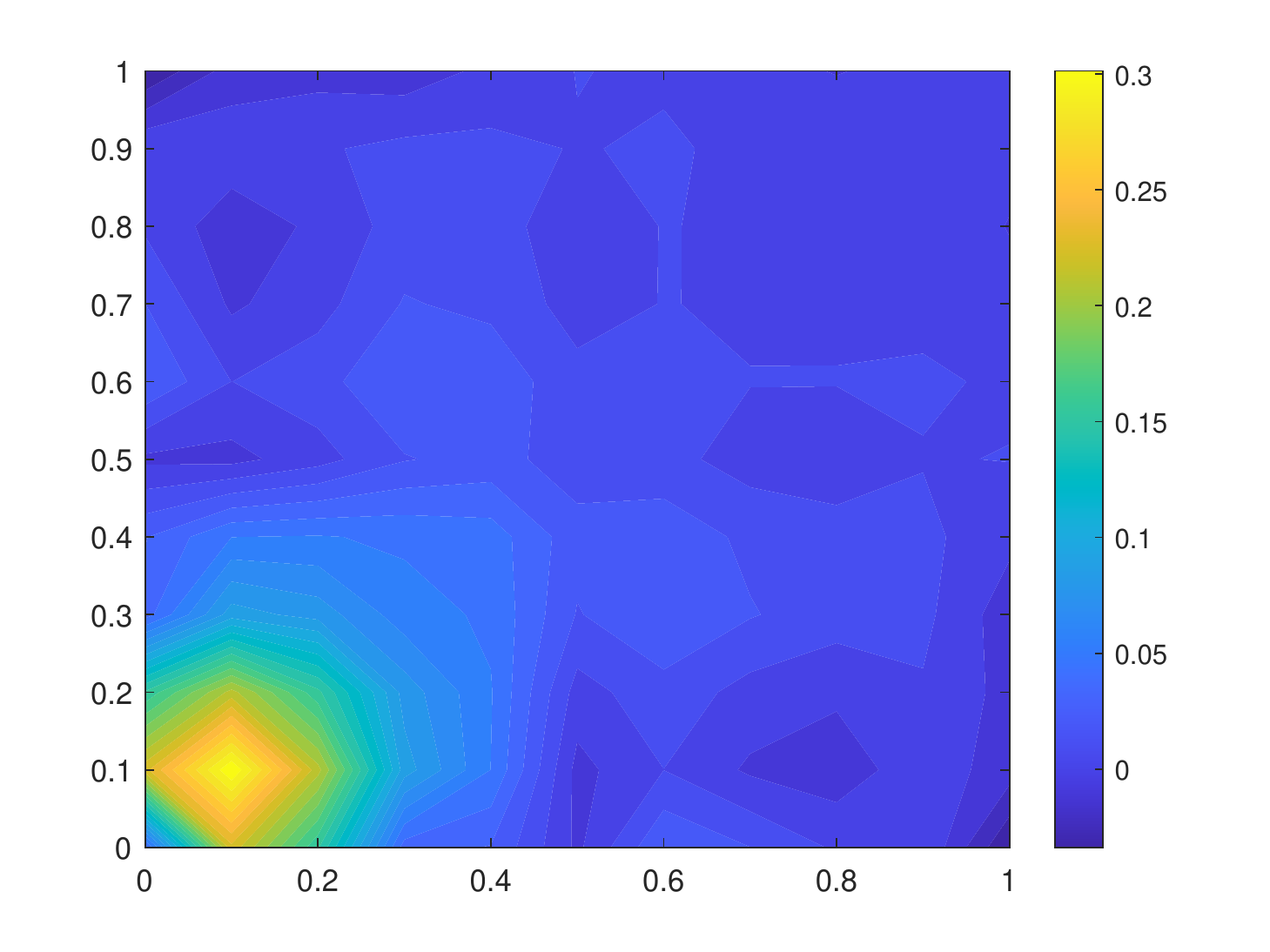}}
  \subfigure[]{\includegraphics[scale=0.35]{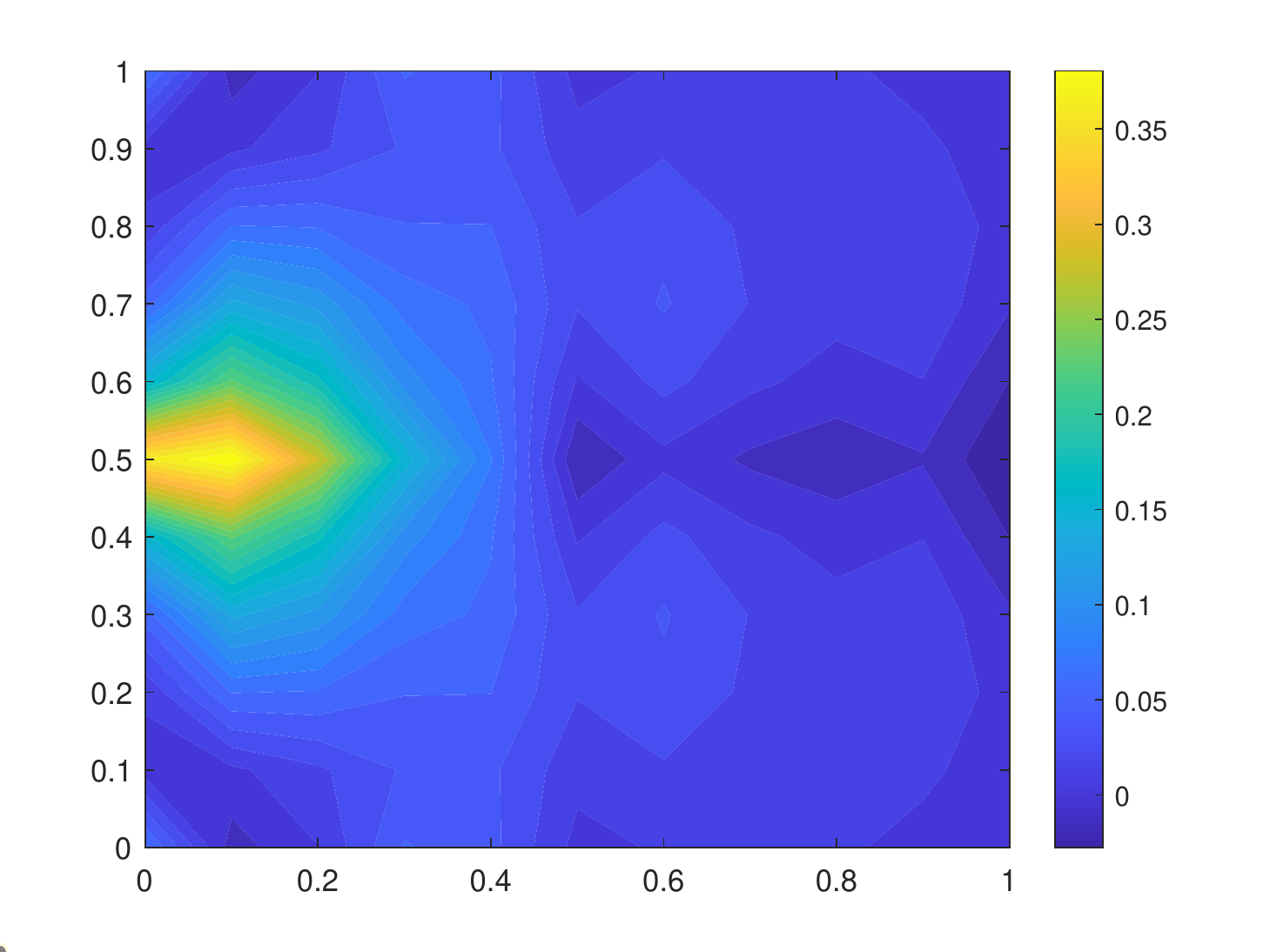}}
  \subfigure[]{\includegraphics[scale=0.35]{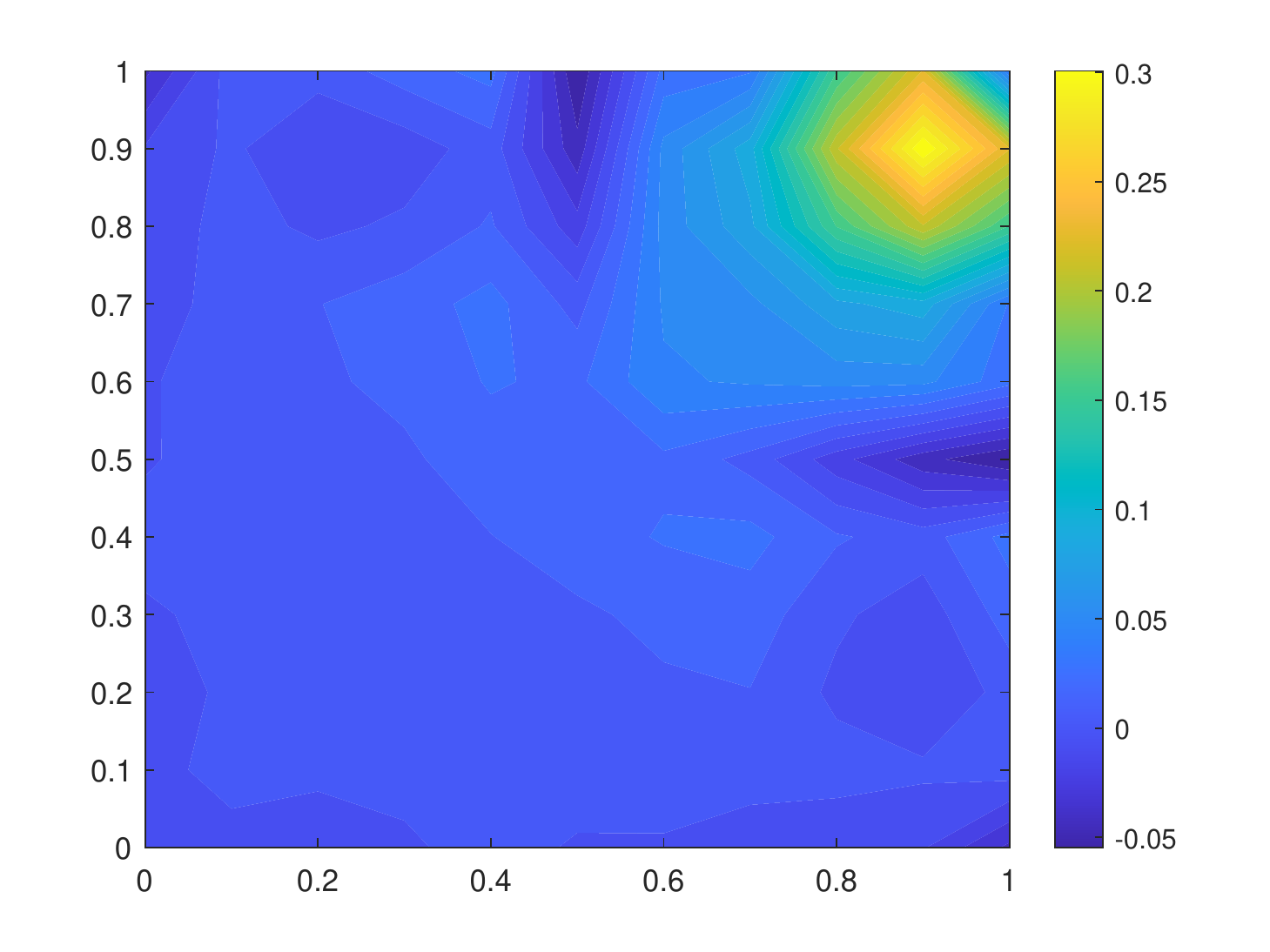}}
  \subfigure[]{\includegraphics[scale=0.35]{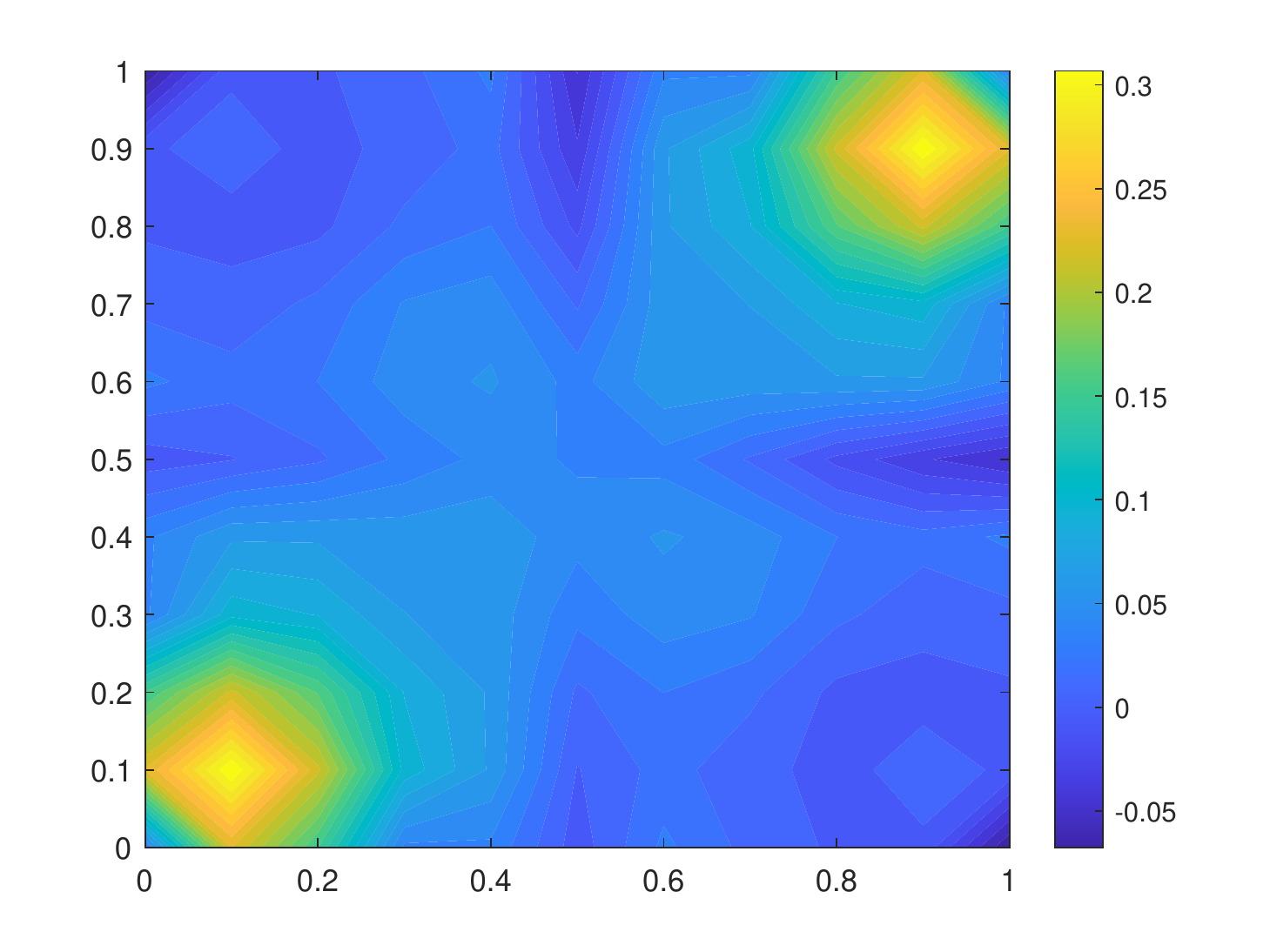}}
  \subfigure[]{\includegraphics[scale=0.35]{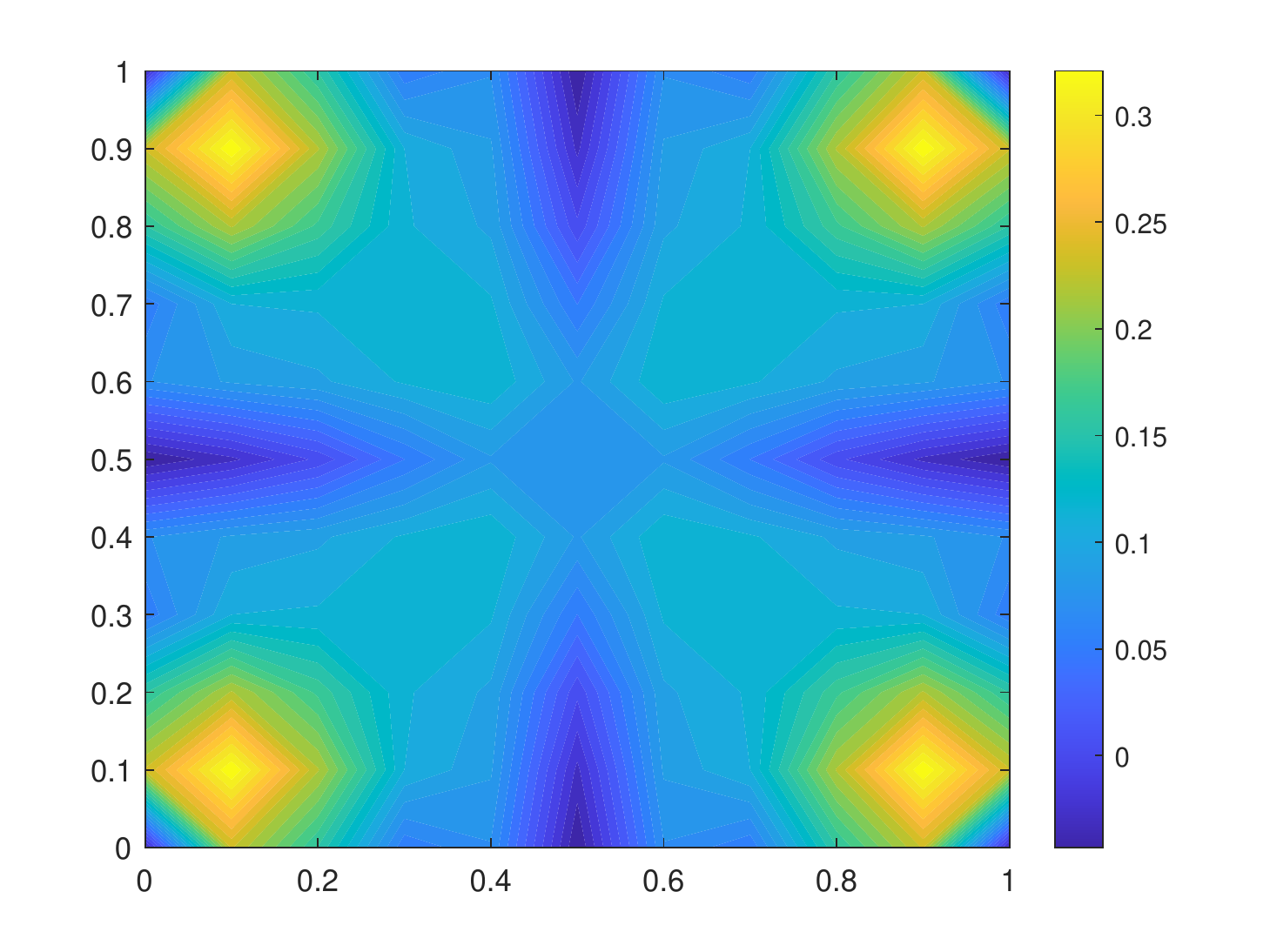}}
  \subfigure[]{\includegraphics[scale=0.35]{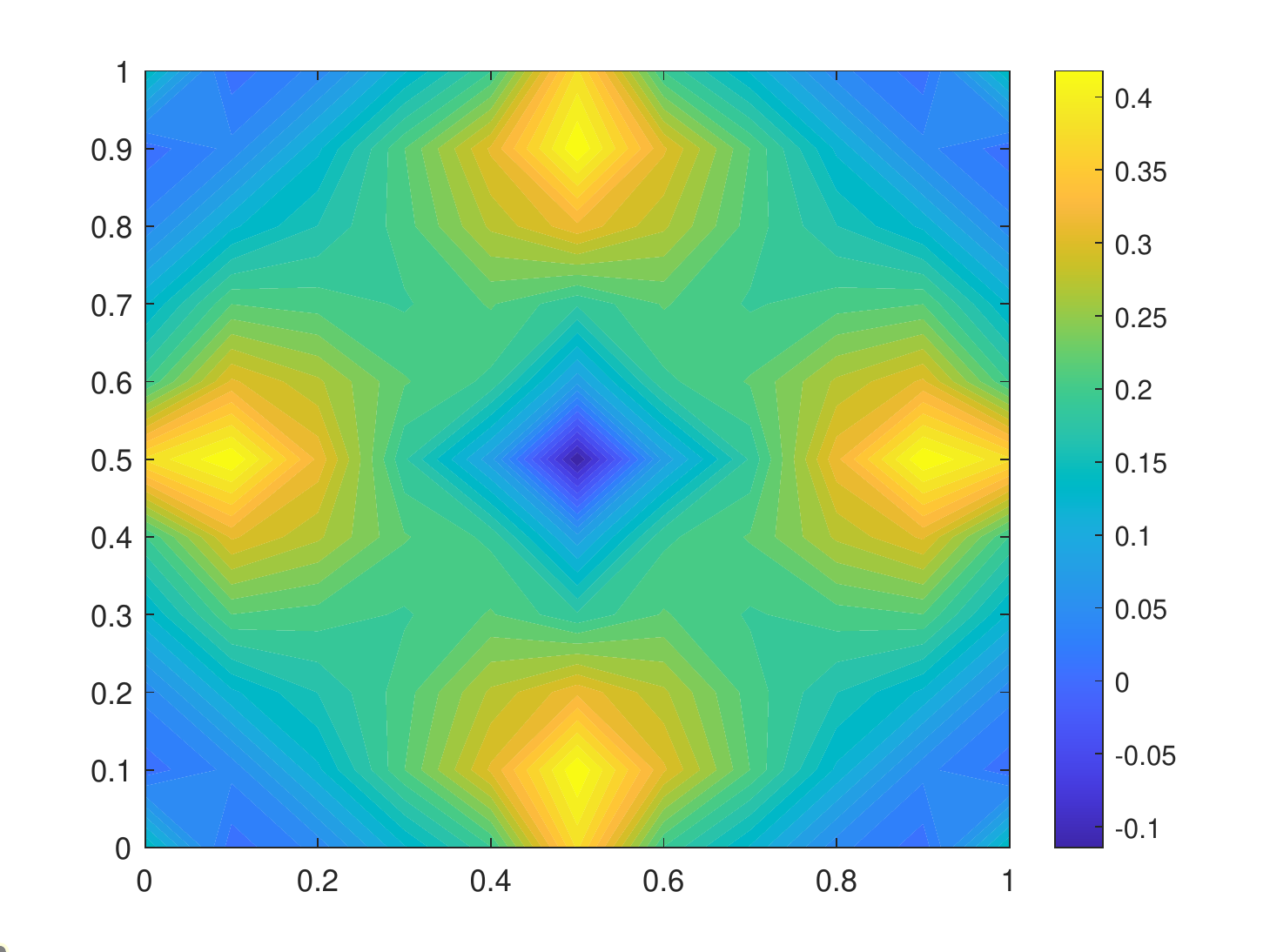}}\\
  \caption{Contour map of the photon flux distributions for Example \ref{ex:source}} \label{Fig:source}
\end{figure}

\subsection{Examples with complex spatial domains in two dimensions}

In the following, we extend the method to solve the RTE for some non-tensor product spatial domains in two dimensions. We always consider the isotropic scattering.

\begin{figure}[!tbh]
  \centering
  \includegraphics[scale=0.5]{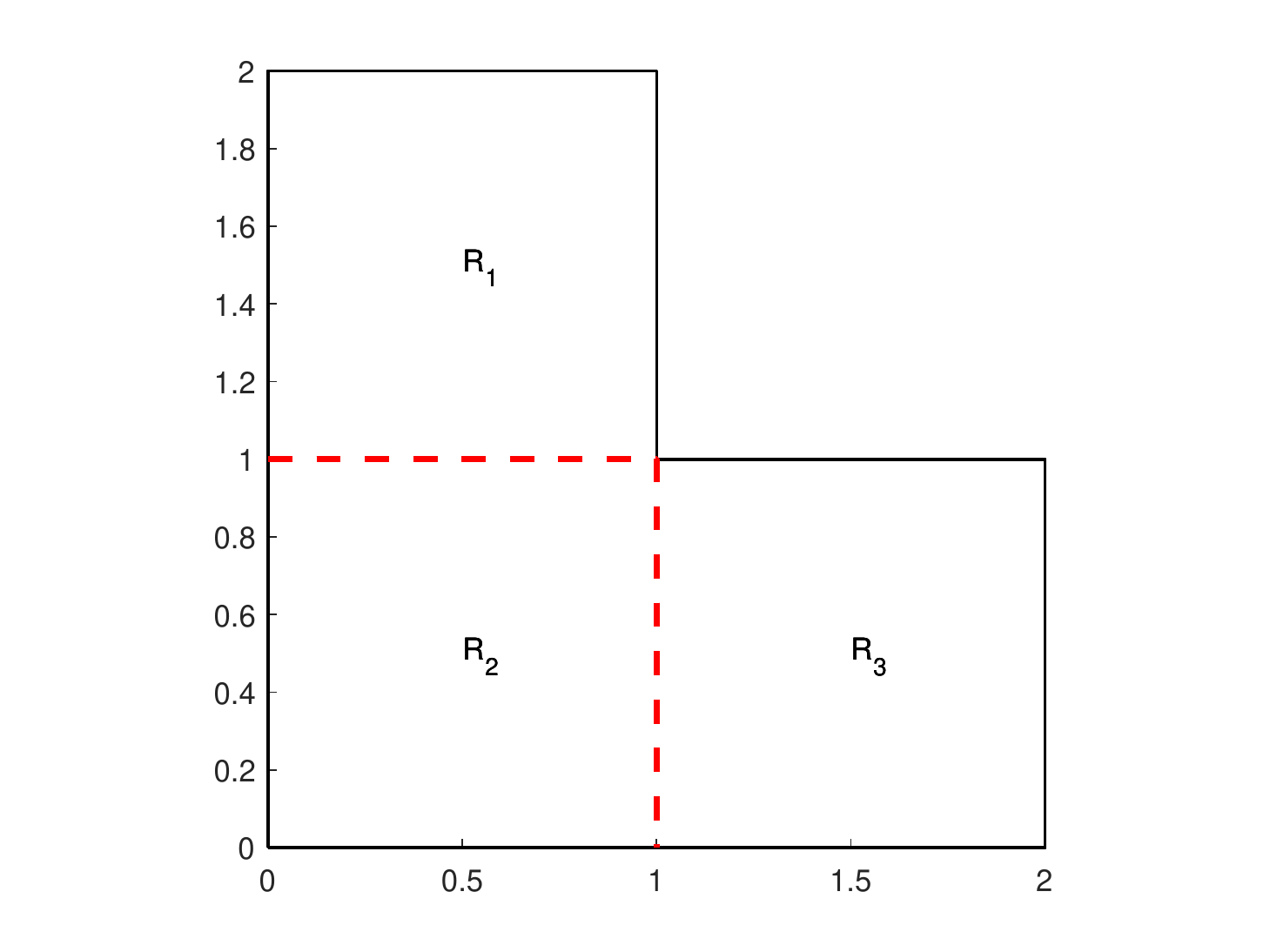}\\
  \caption{Initial subdivision of a $L$-shaped region for Example \ref{ex:3}}\label{Rt}
\end{figure}
\begin{example}\label{ex:3}
The spatial domain $D$ is an $L$-shaped region in 2-D displayed in Fig.~\ref{Rt}, consisting of three rectangles $R_1$, $R_2$ and $R_2$, where
\[R_1 = \Big\{(x_1,x_2): 0 \le x_1 \le 1,~~1 \le x_2 \le 2 \Big\},\]
\[R_2 = \Big\{(x_1,x_2): 0 \le x_1 \le 1,~~0 \le x_2 \le 1 \Big\},\]
\[R_3 = \Big\{(x_1,x_2): 1 \le x_1 \le 2,~~0 \le x_2 \le 1 \Big\}.\]
The parameters are the same as Example \ref{ex:1} and the true solution is
$u(\boldsymbol x,\boldsymbol\omega ) = \sin (\pi x_1)\sin (\pi x_2).$
\end{example}
\begin{figure}[!tbh]
  \centering
  \subfigure[Exact solution]{\includegraphics[scale=0.35]{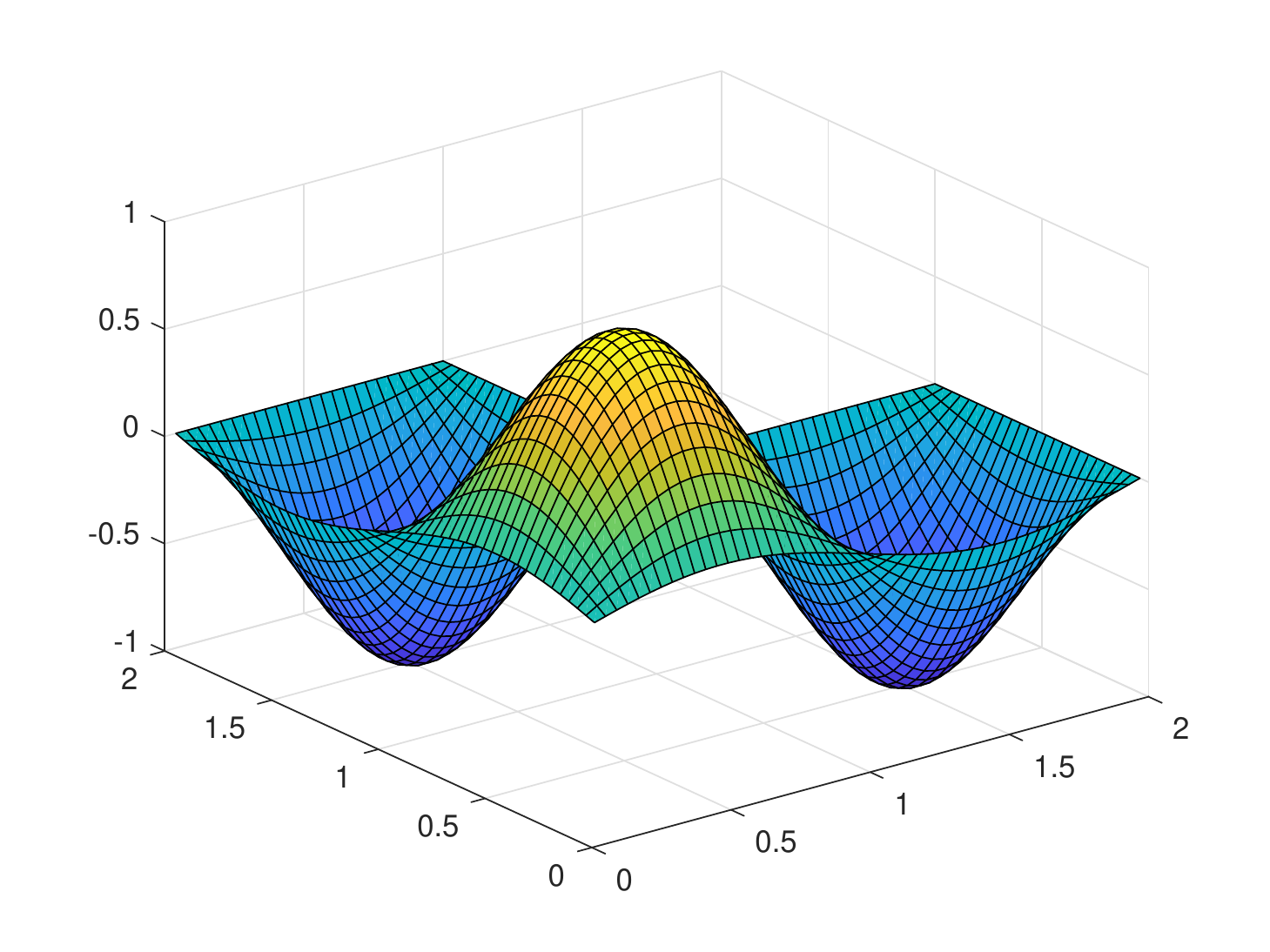}}
  \subfigure[$N=2, k = 1$]{\includegraphics[scale=0.35]{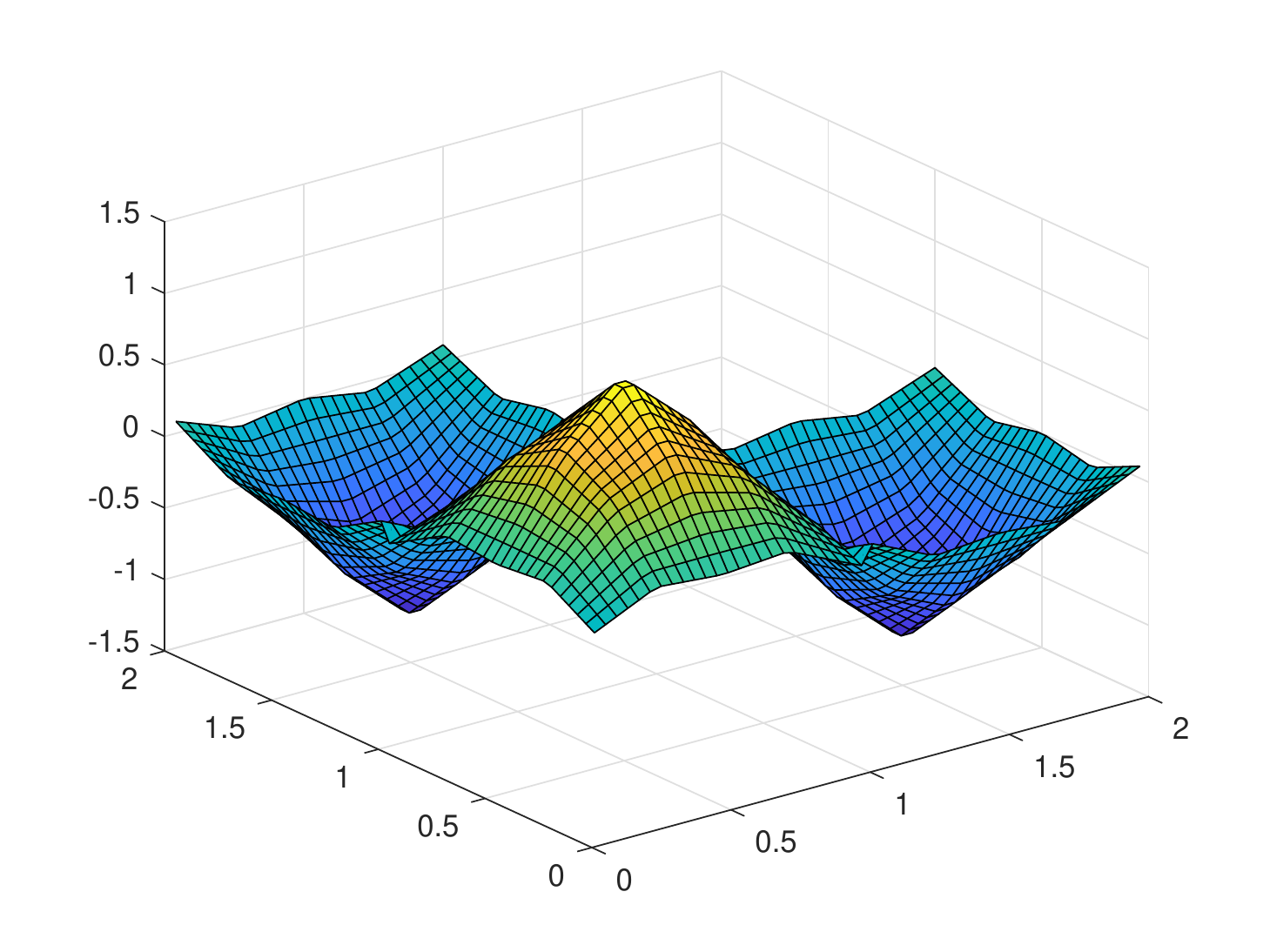}}
  \subfigure[$N=2, k=2$]{\includegraphics[scale=0.35]{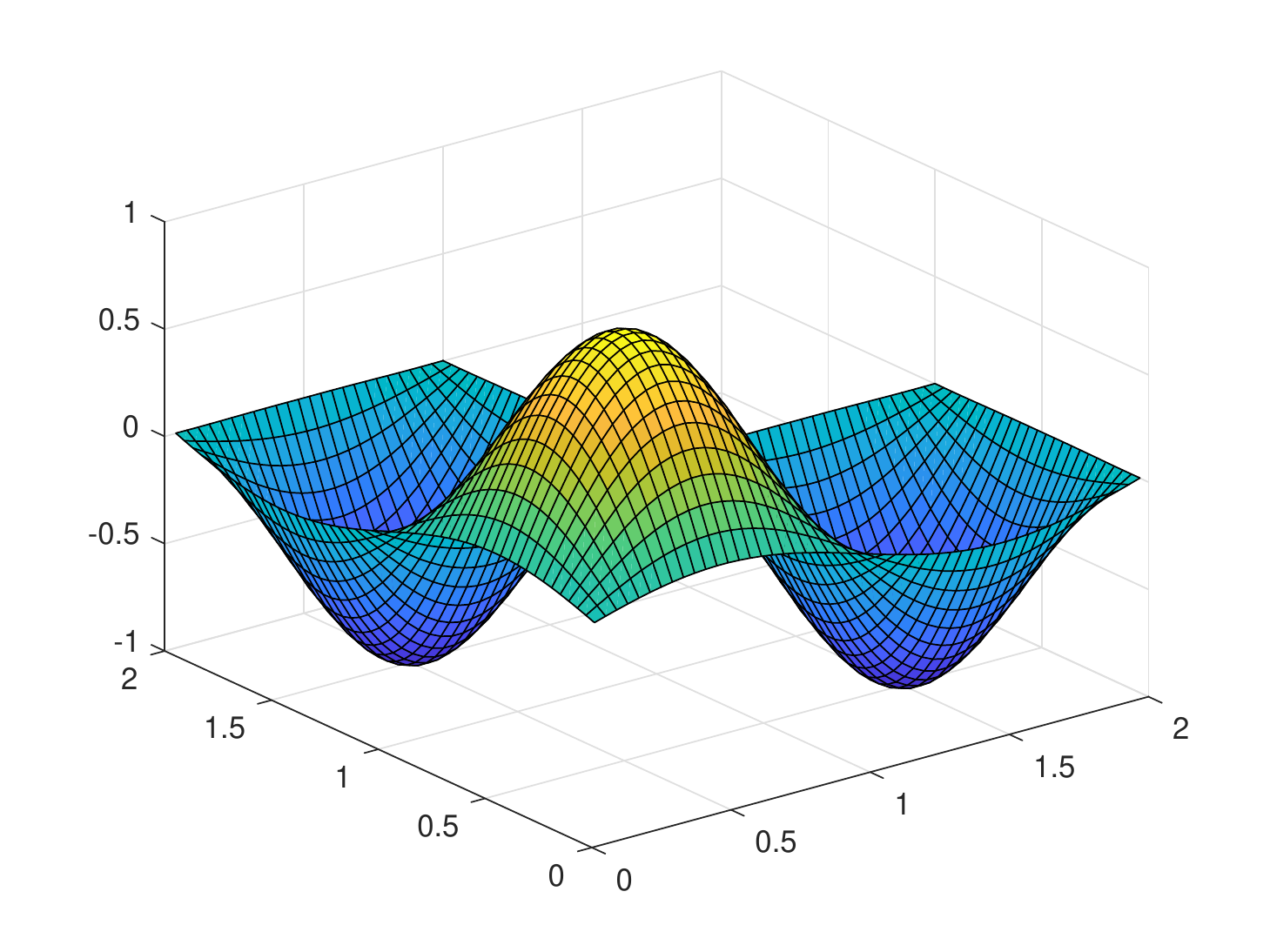}}\\
  \caption{Exact and numerical solutions for Example \ref{ex:3}} \label{Lshapedfig}
\end{figure}

Let $R = R_1 \cup R_2 \cup R_3$ be the initial subdivision of the $L$-shaped region and $\boldsymbol V_0^k(R)$ denote the piecewise polynomial space on $R$. We have the following orthogonal decomposition
\[\boldsymbol V_0^k(R) = \boldsymbol V_0^k(R_1) \oplus \boldsymbol V_0^k(R_2) \oplus \boldsymbol V_0^k(R_3),\]
where functions in $\boldsymbol V_0^k({R_j})~(j=1,2,3)$ are extended by zero to $\mathbb R^2$. For each $R_j$, one can regard it as $[0,1]^2$ and give the sparse representation by using an affine transformation. In Fig.~\ref{Lshapedfig}, we display the numerical solutions for different $N$ and $k$ and the relative errors are given in Tab.~\ref{tab:Lshaped1}.

\begin{table}[!htb]
  \centering
  \caption{Relative errors for Example \ref{ex:3}}\label{tab:Lshaped1}
\begin{tabular}{ccccccccccccccc}
  \hline
    $N$ & $k = 1$ & $k = 2$   & $k = 3$  & $k = 4$ \\ \hline
    1   & 2.2059e-01 & 1.6769e-02 & 1.7691e-03 & 1.3197e-04\\
    2   & 6.1359e-02 & 2.1758e-03 & 1.1360e-04 & 4.1841e-06\\
    3   & 1.7434e-02 & 3.0707e-04 & 7.2792e-06 & 2.1799e-07\\
    4   & 4.8163e-03 & 4.2519e-05 & 4.6425e-07 & -\\
  \hline
\end{tabular}
\end{table}
\begin{figure}[!tbh]
  \centering
  \subfigure[Initial division]{\includegraphics[scale=0.35]{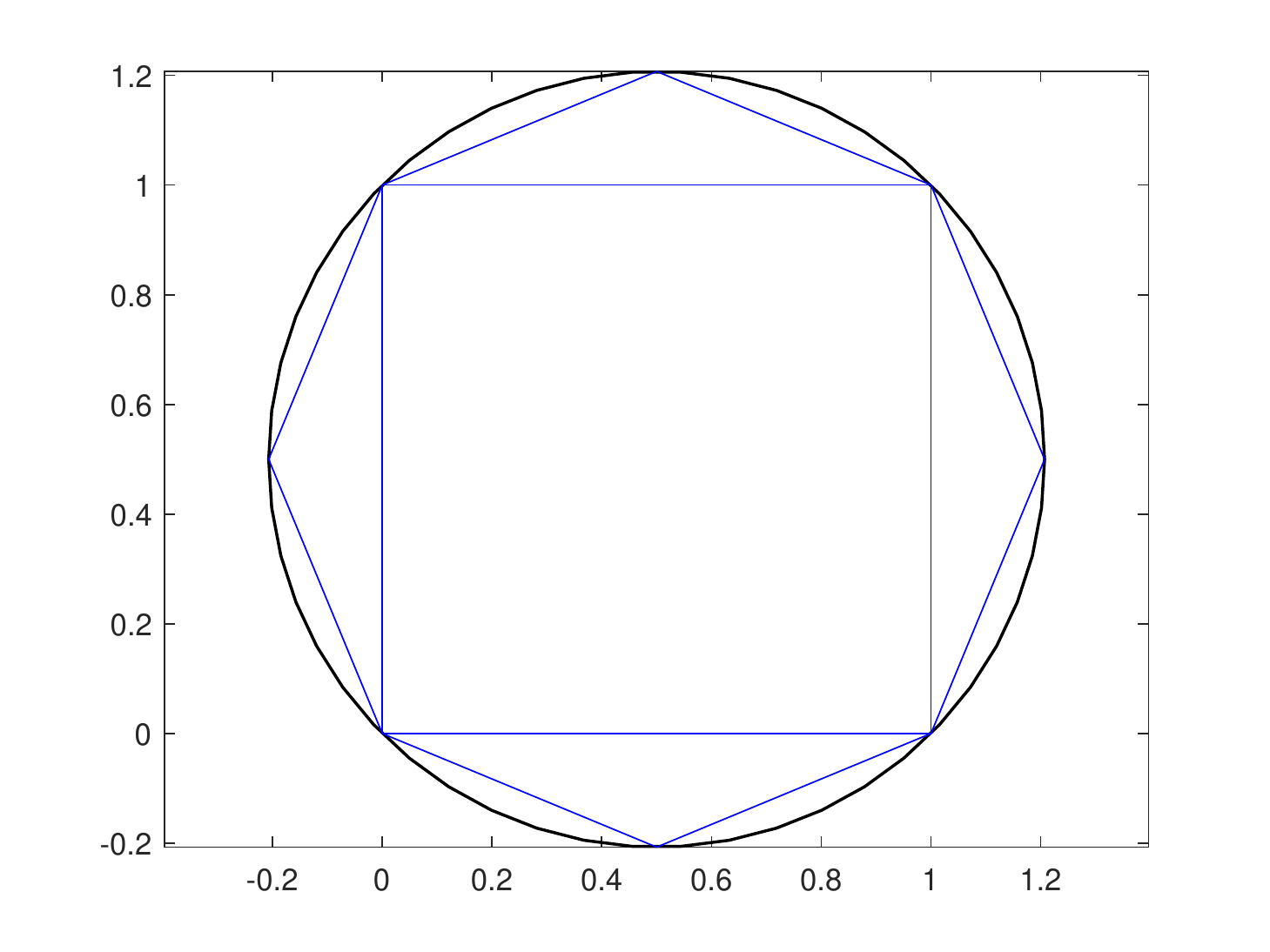}}
  \subfigure[The final division~$\mathcal{T}_h$]{\includegraphics[scale=0.35]{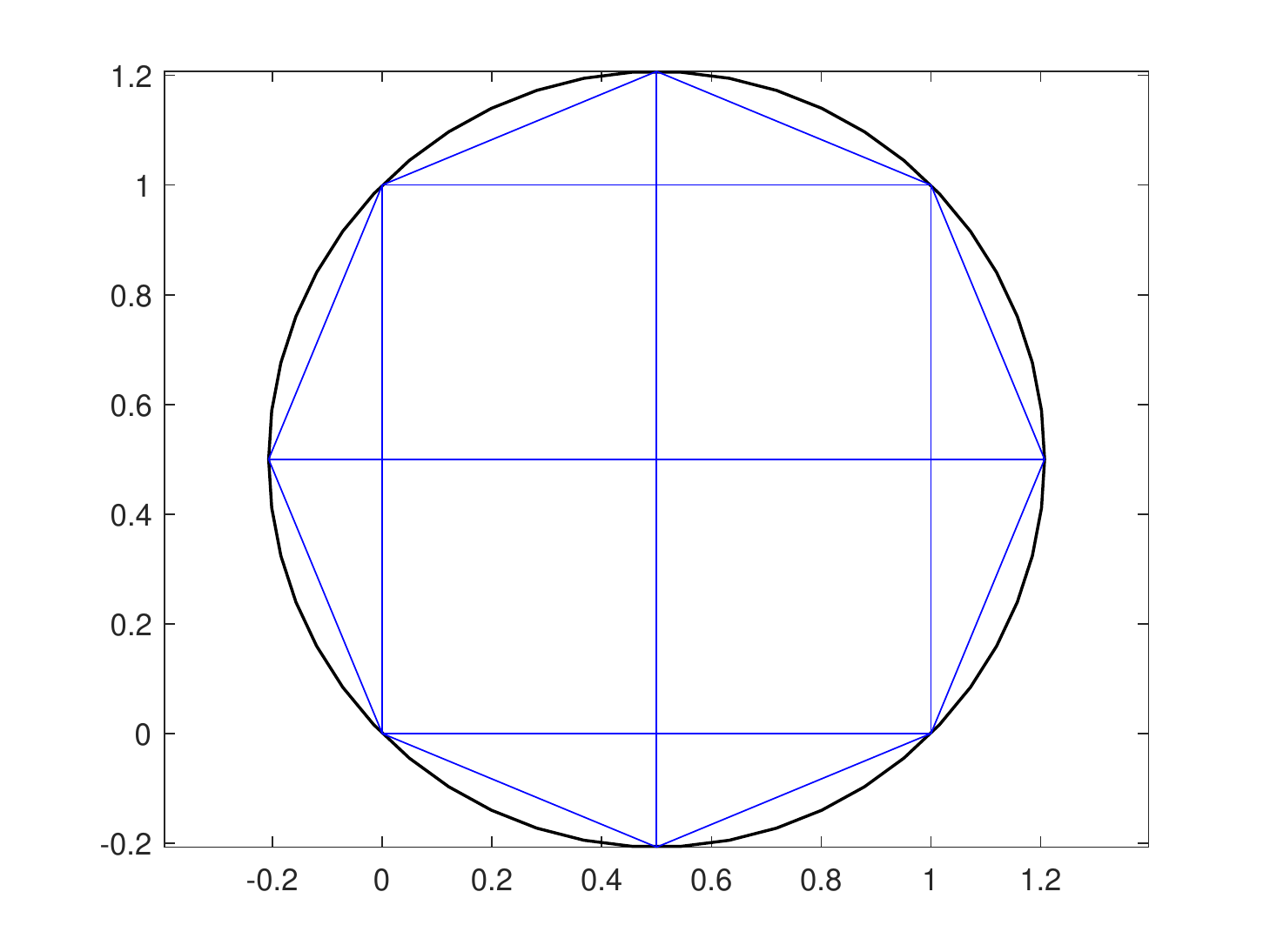}}\\
  \caption{Subdivision of the circular domain}\label{fig:circledom}
\end{figure}

\begin{example}\label{ex:4}
The spatial domain $D$ is a circular region displayed in Fig.~\ref{fig:circledom}. The parameters and the true solution are the same as the last example.
\end{example}
To use the sparse grid method, we first plot a sufficiently large rectangle in the domain and approximate the boundary curve by a polygon as depicted in Fig.~\ref{fig:circledom}~(a). For simplicity, the boundary data corresponding to the polygon is obtained from the exact solution. For the general case, some approximation should be implemented, for example,  the technique from the isoparametric finite elements.  To avoid hanging nodes, the polygon approximation and the corresponding triangulation can be made consistent with the final partition of the rectangle, see Fig.~\ref{fig:circledom}~(b). We should note that the hanging nodes are allowed in our procedure since no interelement continuity is required. Denote the rectangle by $R$ and the other triangles by $T_1,\cdots,T_8$, respectively. Let $\Omega = R\cup {T_1} \cup  \cdots  \cup {T_8}$. We then consider the initial DG space given by
 \[\boldsymbol V_0^k(\Omega) = \boldsymbol V_0^k({R}) \oplus \boldsymbol V_0^k({T_1}) \oplus  \cdots  \oplus \boldsymbol V_0^k({T_8}).\]

 \begin{table}[!htb]
  \centering
  \caption{Orthonormal bases on the reference triangle $\tau$ for $k\le 2$~($r:=\lambda_1, s:=\lambda_2$)}\label{orthobas}
  \begin{tabular}{clllllllllll}
    \hline
    $k=0$    &  \\
             &  $\varphi_1 = \sqrt 2 $\\
    \hline
    $k=1$    &  \\
             &  $\begin{gathered}
  \varphi_2 =  6r-2 \hfill \\
  \varphi_3 = 2\sqrt 3 (2r + s-1) \hfill \\
\end{gathered} $ \\
    \hline
    $k=2$    &  \\
             &  $\begin{gathered}
  \varphi_4 = \sqrt 6 (10r^2 - 8r + 1) \hfill \\
  \varphi_5 = 3\sqrt 2 (5r - 1)(r + 2s - 1) \hfill \\
  \varphi_6 = \sqrt {30} (r^2 + 6rs - 2r + 6s^2 - 6s + 1) \hfill \\
\end{gathered} $ \\
    \hline
  \end{tabular}
\end{table}

The orthonormal bases corresponding to $R$ has been given in the previous section, while the orthonormal bases on each $T_i$ can be obtained by using the Gram-Schmidt procedure. For any triangle $T$ with vertices $z_i = {({x_i},{y_i})}$, $i=1,2,3$, any point $z = {(x,y)}$ can be represented by the barycentric coordinates as
\[
\begin{cases}
  x = x_1\lambda_1 + x_2\lambda_2 + x_3\lambda_3,  \\
  y = y_1\lambda_1 + y_2\lambda_2 + y_3\lambda_3,  \\
  1 = \lambda_1 + \lambda_2 + \lambda_3.
\end{cases}
\]
Note that
\[\iint_T f(x,y)g(x,y){\rm d}x{\rm d}y = 2| T |\int_0^1 \int_0^{1 - \lambda_1} \tilde f(\lambda_1,\lambda_2)\tilde g(\lambda_1,\lambda_2)  {\rm d}\lambda_2{\rm d}\lambda_1,\]
where
\[\tilde f(\lambda_1,\lambda_2) := f(x(\lambda_1,\lambda_2),y(\lambda_1,\lambda_2)).\]
We then define an inner product on the reference triangle $\tau$ by
\[(\tilde f,\tilde g)_\tau = \int_0^1 \int_0^{1 - \lambda_1} \tilde f(\lambda_1,\lambda_2)\tilde g(\lambda_1,\lambda_2)  {\rm d}\lambda_2{\rm d}\lambda_1.\]
Given the orthonormal bases on $\tau$ by $\{\tilde \varphi _i\}$, we then obtain the bases on $T$ given by
\[\psi_i(x,y) = \sqrt {\frac{1}{2| T |}} \varphi_i(x,y).\]
For any function $f(x,y)$ defined on $T$, the projection coefficients are computed as
\begin{align*}
  c_i &= \iint_T f(x,y)\psi_i(x,y){\rm d}x{\rm d}y = 2| T |\int_0^1 \int_0^{1 - \lambda_1} \tilde f(\lambda_1,\lambda_2){{\tilde \psi }_i}(\lambda_1,\lambda_2)  {\rm d}\lambda_2{\rm d}\lambda_1 \\
  & = \sqrt {2| T |} \int_0^1 \int_0^{1 - \lambda_1} \tilde f(\lambda_1,\lambda_2){\tilde \varphi }_i(\lambda_1,\lambda_2) {\rm d}\lambda_2{\rm d}\lambda_1 = :\sqrt {2| T |} {\tilde c}_i.
\end{align*}
The orthogonal bases on $\tau$ are obtained by using Gram-Schmidt procedure to the polynomial set $\{ 1,\lambda_1,\lambda_2,\lambda_1^2,\lambda_1\lambda_2,\lambda_2^2, \cdots  \}$, some of which are listed in Tab.~\ref{orthobas}.

The relative error is defined by ${\rm Err } = \| f - f_h \|_{L^2(\mathcal{T}_h)}/\| f \|_{L^2(\mathcal{T}_h)}$ and given in Tab.~\ref{tab:circl}.
\begin{table}[!htb]
  \centering
  \caption{Relative errors for Example \ref{ex:4}~$(N=2)$}\label{tab:circl}
\begin{tabular}{ccccccccccccccc}
  \hline
    $k$ & 0 & 1   & 2  & 3 \\ \hline
    Err   & 5.8510e-01 & 6.1678e-02 & 9.3273e-03 & 5.5965e-04\\
  \hline
\end{tabular}
\end{table}

Summarizing our main observations from the numerical results reported in all previous examples, we may conclude that
\begin{itemize}
  \item The sparse discrete ordinate DG method can greatly reduce the spatial degrees of freedom while keeping almost the same accuracy up to multiplication of an log factor.
  \item The proposed method is highly effective for problems away from strong forward scattering. To get an improved result, large discrete-ordinate sets are needed for highly forward-peaked case.
  \item The method can be extended to solve the RTE efficiently for some non-tensor product spatial domains in two dimensions.
\end{itemize}

\section{Conclusions and remarks}

In this paper, we combine the sparse grid technique with the discrete ordinate DG method to solve the RTE with inflow boundary conditions, which can be adapted to other types of boundary conditions. Under suitable regularity assumptions, we derive error estimates for the numerical solutions. Results from many numerical examples show the good convergence behavior of the method. For highly forward-peaked scattering, there have been substantial efforts made to develop simpler approximations to integral scattering operator $S$. One well-established example is the so-called Fokker-Planck equation (cf. \cite{Sheng-Han-2013}), to which the sparse grid techniques can also be applied.

\section*{Acknowledgments}

The work was partially supported by NSFC (Grant No. \ 12071289) and the Strategic Priority Research Program of Chinese Academy of Sciences (Grant No. \ XDA25010402).



\end{document}